\documentclass[11pt]{article}
\usepackage[T1]{fontenc}
\usepackage{lmodern}
\usepackage{amsmath,amssymb,amsthm,mathtools,mathrsfs}
\usepackage[a4paper,margin=28mm]{geometry}
\usepackage{microtype,booktabs,array}
\usepackage{graphicx,tikz}
\usetikzlibrary{arrows.meta,positioning}
\usepackage[colorlinks=true,linkcolor=blue!45!black,citecolor=blue!45!black,urlcolor=blue!45!black]{hyperref}
\usepackage{enumitem}
\setlist{itemsep=3pt,topsep=5pt}
\numberwithin{equation}{section}
\newtheorem{theorem}{Theorem}[section]
\newtheorem{proposition}[theorem]{Proposition}
\newtheorem{lemma}[theorem]{Lemma}
\newtheorem{corollary}[theorem]{Corollary}
\theoremstyle{definition}
\newtheorem{definition}[theorem]{Definition}
\newtheorem{example}[theorem]{Example}
\newtheorem{remark}[theorem]{Remark}
\DeclareMathOperator{\Ind}{Ind}

\DeclareMathOperator{\Irr}{Irr}
\DeclareMathOperator{\Lin}{Lin}
\DeclareMathOperator{\Supp}{Supp}
\DeclareMathOperator{\Stab}{Stab}
\DeclareMathOperator{\sgn}{sgn}
\DeclareMathOperator{\rank}{rank}
\DeclareMathOperator{\Hom}{Hom}
\DeclareMathOperator{\GL}{GL}
\DeclareMathOperator{\AGL}{AGL}
\DeclareMathOperator{\Sp}{Sp}
\DeclareMathOperator{\ch}{ch}
\DeclareMathOperator{\Ker}{Ker}
\newcommand{\Z}{\mathbb Z}
\newcommand{\Q}{\mathbb Q}
\newcommand{\C}{\mathbb C}
\newcommand{\F}{\mathbb F}
\newcommand{\A}{\mathcal A}
\newcommand{\HH}{\mathcal H}
\newcommand{\QQ}{\mathfrak Q}
\newcommand{\ee}{\varepsilon}
\newcommand{\dd}{\delta}
\newcommand{\om}{\omega}
\newcommand{\emptypar}{\varnothing}

\newcommand{\row}{\mathrm{row}}
\newcommand{\col}{\mathrm{col}}

\newcommand{\tr}{\mathrm{tr}}

\newcommand{\one}{\mathbf 1}
\newcommand{\transpose}{\mathsf t}
\newcommand{\pair}{\mathsf P}
\newcommand{\diagauto}{\diamond}
\title{Quasiparabolic Gelfand models\\
for finite irreducible Coxeter groups\\
and their canonical Hecke structures}
\author{Yifeng Zhang
\\
School of Mathematical Sciences \\
South China Normal University \\
{\tt calvinz314159\@gmail.com}}
\date{}
\hypersetup{pdftitle={Quasiparabolic Gelfand models for finite irreducible Coxeter groups and their canonical Hecke structures},pdfauthor={Yifeng Zhang}}
\begin{document}
\begingroup\renewcommand{\@}{\char64\relax}\maketitle\endgroup
\begin{abstract}
Quasiparabolic sets extend parabolic coset spaces while retaining a length filtration and natural Hecke algebra deformations.
Gelfand models built from induced linear characters ask when such spaces can account for every irreducible representation exactly once.
We classify, up to equality of the individual induced characters, all quasiparabolic Gelfand models for finite irreducible Coxeter groups with inducing characters restricted from their ambient parabolic subgroups, including the additional models of type \(D_{4r+2}\) and \(B_3\).
We derive the associated finite Gelfand-pair and commutant consequences, construct canonical Hecke structures for the additional classical models.
\end{abstract}
\noindent\textit{Keywords:} Coxeter group; quasiparabolic set; Gelfand model; Gelfand pair; Hecke algebra; canonical basis; \(W\)-graph.

\section{Introduction}
\label{sec:intro}

A Gelfand model of a finite group is a complex representation containing each irreducible representation with multiplicity one.
The problem of constructing such a representation from simple combinatorial data has a substantial history.
Klyachko's models for finite general linear groups \cite{Klyachko} and the explicit symmetric-group model of Inglis, Richardson, and Saxl \cite{IRS} are basic examples.
Baddeley \cite{Baddeley} and Vinroot \cite{Vinroot} studied models arising from involutions in Weyl and Coxeter groups.
Combinatorial constructions for symmetric groups, wreath products, and classical Weyl groups were developed by Adin, Postnikov, and Roichman \cite{APR,APRwreath}, Caselli and Fulci \cite{CaselliFulci}, and Araujo and Bratten \cite{AraujoBratten}.
These constructions connect induction of linear characters with the partition combinatorics of irreducible representations.
We use the symmetric-function conventions of Macdonald \cite{MacdonaldBook} and the representation-theoretic conventions of Geck and Pfeiffer \cite{GeckPfeiffer}.

At the level of a single inducing term, the relevant notion is a finite Gelfand pair or, for a nontrivial inducing character, a linear Gelfand triple.
The associated double-coset algebra is commutative exactly when the induced character is multiplicity-free; its primitive idempotents and spherical functions encode the irreducible support \cite{StembridgeQ,CSST}.
Godsil and Meagher classified the multiplicity-free permutation representations of symmetric groups \cite{GodsilMeagher}, while Turek classified the subgroups admitting multiplicity-free induction from an irreducible character in sufficiently large degree and determined most possible inducing characters \cite{Turek}.
For Coxeter groups, Abramenko, Parkinson, and Van Maldeghem classified the commutative parabolic Hecke algebras \cite{APVM}.
These results provide necessary candidate filters and independent comparisons below, but the present problem also permits nonparabolic QP stabilizers and imposes an exact cover of all irreducible characters.

Hecke algebras provide a second source of structure.
Kazhdan and Lusztig \cite{KL} introduced their canonical basis and \(W\)-graphs, while Deodhar \cite{Deodhar} developed the parabolic analogue.
Howlett and Yin \cite{HY1,HY2} established induction procedures for \(W\)-graphs.
The involution modules of Lusztig and Vogan \cite{LV} and Lusztig \cite{Lusztig} show that bar operators and canonical bases extend beyond ordinary parabolic permutation modules.
Rains and Vazirani \cite{RV} introduced quasiparabolic sets, a class of scaled Coxeter-group sets with a Bruhat order and two natural Hecke modules.
Lu \cite{Lu} classified quasiparabolic subgroups in finite classical types and proved the existence of bar operators for their ordinary Hecke modules.
Marberg \cite{MarbergBar} established the canonical-basis and \(W\)-graph formalism associated with such bar operators.
The perfect models of Marberg and Zhang \cite{MZgraphs,MZperfect} give a particularly structured source of Gelfand \(W\)-graphs.
Their classification concerns quasiparabolic centralizers of perfect involutions, with inducing characters subject to a prescribed parabolic extension condition.
The present paper keeps that extension condition and allows arbitrary quasiparabolic stabilizers.
Accordingly, their existence results provide the initial families, whereas completeness in the enlarged class requires a further subgroup argument.

Other quasiparabolic subgroups can contribute new induced characters.
Green and Xu \cite{GXroots} constructed quasiparabolic sets of maximal orthogonal positive-root sets, extending the setting of Macdonald's root-product representations \cite{MacdonaldRoots}.
Their subsequent work gives generalized Rothe diagrams and explicit descriptions in exceptional types \cite{GXRothe,GXFano}.
These examples motivate allowing quasiparabolic stabilizers beyond perfect-involution centralizers.
For Hecke realizations, we distinguish a character-level construction from one preserving a specified integral lattice.
Gyoja's existence theorem \cite{Gyoja}, in the formulation of Hahn \cite{Hahn}, supplies \(W\)-graphs over a splitting field; it does not identify the natural induction lattice.
We work throughout with a fixed class of inducing data.

Here is the precise class that we classify.
For a finite Coxeter system \((W,S)\), let \(W_J=\langle J\rangle\), where \(J\subseteq S\), and let \(\Lin(W_J)\) denote its one-dimensional complex characters.
A subgroup \(H\leq W_J\) is \emph{quasiparabolic} when its coset space, with the minimum-length height, satisfies the two axioms recalled in Definition~\ref{def:qp}.
An allowed character is
\[
 \chi_{J,H,\sigma}=\Ind_H^W(\sigma|_H),
 \qquad H\leq W_J\text{ quasiparabolic},\quad \sigma\in\Lin(W_J).
\]
Write \(\Gamma_W=\sum_{\rho\in\Irr(W)}\rho\), where \(\Irr(W)\) is the set of irreducible complex characters.
A \emph{quasiparabolic Gelfand model} is an unordered family of allowed characters whose sum is \(\Gamma_W\).
Two families are equivalent if their individual induced characters agree after a reordering.
Thus the equivalence retains the decomposition into inducing terms.
We write \(\QQ(W)\) for the set of equivalence classes.
The natural lattice of a term is
\(\Z W\otimes_{\Z H}\Z_{\sigma|_H}\), where \(H\) acts on the rank-one factor by \(\sigma|_H\).
The terms bar operator, canonical basis, and \(W\)-graph have the meanings made precise in Section~\ref{sec:prelim}; canonicality always refers to a specified ordered standard basis and bar operator.

\begin{theorem}[Classification and Hecke realization]
\label{thm:main}
For finite irreducible Coxeter groups, the complete numbers of quasiparabolic Gelfand models are as follows:
\[
\begin{array}{c|c}
 W&|\QQ(W)|\\ \hline
 A_1&2\\
 A_3&4\\
 A_{n-1},\ n=3\text{ or }n\geq5&2\\
 B_3&24\\
 B_n,\ n\geq4&4\\
 D_n,\ n\geq5\text{ odd}&2\\
 D_n,\ n\equiv2\pmod4,\ n\geq6&4\\
 D_n,\ 4\mid n&0\\
 I_2(m),\ m\geq3\text{ odd}&2\\
 I_2(m),\ 4\mid m&8\\
 I_2(m),\ m\equiv2\pmod4&16\\
 H_3&4\\
 H_4,F_4,E_6,E_7,E_8&0 .
\end{array}
\]
The coincident types \(A_2=I_2(3)\) and \(B_2=I_2(4)\) are interpreted consistently.
The families themselves are given in Theorems~\ref{thm:A}, \ref{thm:B}, \ref{thm:D}, \ref{thm:dihedral}, \ref{thm:H3}, and \ref{thm:B3list}.

Every model has a Hecke realization and a \(W\)-graph over a splitting field.
In types \(A\), \(B_n\) with \(n\geq4\), and every type \(D_n\) admitting a model, the classified families have realizations on natural induction lattices with bar operators, canonical bases, and integral \(W\)-graphs.
All twenty-four type \(B_3\) families have such natural Hecke and canonical-basis realizations; exactly fourteen have termwise natural canonical \(W\)-graphs.
The other ten contain a natural summand lattice which is not the specialization of any integral \(W\)-graph module.
\end{theorem}

This paper is organized as follows.
Section~\ref{sec:prelim} records the definitions and literature results used in the proofs.
Sections~\ref{sec:A}--\ref{sec:D} establish the classifications in types \(A\), \(B\), and \(D\); the complete type \(B_3\) list appears in Subsection~\ref{subsec:B3classification}.
Section~\ref{sec:exceptional} treats the dihedral and exceptional types.
Section~\ref{sec:hecke} constructs the Hecke realizations, bar operators, canonical bases, and \(W\)-graphs, with the additional \(B_3\) constructions and integral obstruction in Subsection~\ref{sec:B3canonical}.
\section{Preliminary}
\label{sec:prelim}

\subsection{Coxeter groups and quasiparabolic sets}

A Coxeter system \((W,S)\) has presentation
\[
 W=\langle S:(st)^{m_{st}}=1\rangle,
 \qquad m_{ss}=1,\quad m_{st}=m_{ts}\geq2\ (s\ne t).
\]
Throughout, \(W\) is finite unless explicitly stated otherwise.
Its rank is \(|S|\), its length function is \(\ell\), and its reflections form
\(T=\{wsw^{-1}:w\in W,\ s\in S\}\).
Its Coxeter sign is \(\sgn_W(w)=(-1)^{\ell(w)}\).
For \(J\subseteq S\), \(W_J\) is the standard parabolic subgroup and \(W^J\) is the set of minimum-length representatives of \(W/W_J\).
A conjugate of a standard parabolic subgroup is called parabolic.
We use the usual real reflection representation \(V\), as in \cite{Humphreys,BB}.
A subgroup has \emph{full support} if it is contained in no proper parabolic subgroup.
Equivalently, \(V^H=0\): a nonzero fixed vector may be moved into the closed fundamental chamber, where its stabilizer is a proper standard parabolic subgroup.
The \emph{parabolic closure} of \(H\) is the pointwise stabilizer of \(V^H\).

\begin{definition}[Rains--Vazirani \cite{RV}]
\label{def:qp}
A \emph{scaled \(W\)-set} is a \(W\)-set \(X\) with a height \(h:X\to\Z\) such that
\(|h(sx)-h(x)|\leq1\) for \(s\in S\).
It is \emph{quasiparabolic} if, for \(x\in X\), \(r\in T\), and \(s\in S\),
\begin{align}
 h(rx)=h(x)&\ \Longrightarrow\ rx=x, \tag{QP1}\label{eq:QP1}\\
 h(rx)>h(x),\ h(srx)<h(sx)&\ \Longrightarrow\ rx=sx. \tag{QP2}\label{eq:QP2}
\end{align}
A subgroup \(H\leq W\) is quasiparabolic if \(W/H\), with
\(h(wH)=\min_{u\in wH}\ell(u)\), is quasiparabolic.
The quasiparabolic Bruhat order is generated by \(x<rx\) whenever \(h(x)<h(rx)\).
\end{definition}

We abbreviate quasiparabolic by QP.
A minimum point is a point of least height.
For a subgroup \(H\), put \(H^\circ=H\cap\Ker(\sgn_W)\); the notation always refers to Coxeter parity.
A \emph{rotation} is a product of two distinct reflections.
A Coxeter homomorphism here means a homomorphism taking simple generators to simple generators or the identity, with its image equipped with the indicated Coxeter structure.

\begin{proposition}[Structural facts \cite{RV,Lu}]
\label{prop:QPfacts}
The following facts hold.
\begin{enumerate}[label=(\roman*)]
\item A finite transitive QP set has a unique minimum \(x_0\), and
\[
 h(x)-h(x_0)=\min\{\ell(w):wx_0=x\}.
\]
\item Restriction to a standard parabolic subgroup preserves QP sets, orbit by orbit; parabolic induction preserves QP sets, with height \(\ell(w)+h(x)\) on \(W^J\times X\).
Images under the Coxeter quotient maps used below preserve QP subgroups.
\item If \(H\) is QP, so is \(H^\circ\). Every even QP subgroup is generated by rotations.
An odd QP subgroup contains a simple reflection.
If \(H\) contains a simple reflection, then \(H\) is QP if and only if \(H^\circ\) is QP.
\item In a signed permutation group, let \(j\) be the largest signed symbol in an \(H\)-orbit.
Taking its stabilizer, deleting \(j,-j\), and order-preservingly relabelling the remaining symbols gives a QP subgroup of the corresponding smaller standard parabolic group.
\end{enumerate}
\end{proposition}

Parts (i)--(ii) are the minimum-point, restriction, induction, and quotient results of \cite{RV}.
For (iii), we use \cite[Theorems 1--2 and Chapter 2]{Lu}; part (iv) is the minimum-double-coset operation in \cite[Example 8]{Lu}.
The minimum-point condition is essential when using a subgroup embedding.
Conjugate subgroups have the same permutation character, but their minimum-length coset heights need not both be QP.

\subsection{Characters, partitions, and induction}

Characters are complex characters.
We use
\(\langle\alpha,\beta\rangle_G=|G|^{-1}\sum_{g\in G}\alpha(g)\overline{\beta(g)}\).
A character is \emph{multiplicity-free} (MF) when every irreducible multiplicity is at most one.
Its support is
\(\Supp(\alpha)=\{\rho\in\Irr(G):\langle\alpha,\rho\rangle_G>0\}\).
For a fixed ordering \(\rho_1,\ldots,\rho_d\) of \(\Irr(W)\), an MF character is identified with its support vector in \(\{0,1\}^d\).
An \emph{exact cover} is a family of these supports partitioning \(\{1,\ldots,d\}\).
Thus exact covers are precisely the individual-character models considered here.
An external tensor product is denoted by \(\boxtimes\).
If \(p:G\to Q\) is onto, inflation means composition of a character of \(Q\) with \(p\).
We use Frobenius reciprocity, induction in stages, the index-two induction formula, and Clifford theory in the standard form \cite{Serre,GeckPfeiffer}.
For example, if \(H=\Ker(\eta:L\to\{\pm1\})\) has index two, then
\[
 \Ind_H^G(\tau|_H)=\Ind_L^G\tau+\Ind_L^G(\tau\eta).
\]
All these summands are genuine characters with nonnegative multiplicities.

For \(H\leq G\) and \(\lambda\in\Lin(H)\), the triple
\((G,H,\lambda)\) is a \emph{linear Gelfand triple}, or a
\emph{twisted Gelfand pair}, if \(\Ind_H^G\lambda\) is MF.
When \(\lambda=\one\), this says that \((G,H)\) is a
\emph{Gelfand pair}.
Put
\[
 e_\lambda=\frac1{|H|}\sum_{h\in H}\lambda(h^{-1})h,
 \qquad
 \mathscr A(G,H,\lambda)=e_\lambda\C G e_\lambda .
\]
We call \(\mathscr A(G,H,\lambda)\) the finite twisted double-coset
algebra.  It is distinct from the Iwahori--Hecke algebra \(\HH(W)\)
introduced later.

For \(g\in G\), set \(H_g=H\cap gHg^{-1}\) and
\(\lambda^g(x)=\lambda(g^{-1}xg)\) for \(x\in H_g\).
Call the double coset \(HgH\) \emph{\(\lambda\)-admissible} when
\(\lambda|_{H_g}=\lambda^g|_{H_g}\).

\begin{proposition}[Gelfand-pair criteria \cite{CSST,RSW}]
\label{prop:Gelfandcriteria}
Let \(H\leq G\) and \(\lambda\in\Lin(H)\).
The following statements hold.
\begin{enumerate}[label=(\roman*)]
\item There is an anti-isomorphism
\[
 \mathscr A(G,H,\lambda)
 \longrightarrow
 \operatorname{End}_{\C G}(\Ind_H^G\lambda).
\]
If
\(\Ind_H^G\lambda=\sum_{\rho\in\Irr(G)}m_\rho\rho\), then
\[
 \mathscr A(G,H,\lambda)
 \cong\prod_{m_\rho>0}\operatorname{Mat}_{m_\rho}(\C).
\]
In particular, \((G,H,\lambda)\) is a linear Gelfand triple if and
only if \(\mathscr A(G,H,\lambda)\) is commutative.
\item Mackey theory gives
\begin{equation}
 \left\langle\Ind_H^G\lambda,\Ind_H^G\lambda\right\rangle_G
 =\sum_{HgH\in H\backslash G/H}
   \dim\Hom_{H_g}(\lambda|_{H_g},\lambda^g|_{H_g}).
 \label{eq:MackeyGelfand}
\end{equation}
Thus this inner product, the dimension of \(\mathscr A(G,H,\lambda)\),
and the number of \(\lambda\)-admissible double cosets are equal.
\item Suppose that \(\lambda(H)\subseteq\{\pm1\}\).  If every
\(\lambda\)-admissible double coset contains an involution, then
\((G,H,\lambda)\) is a linear Gelfand triple.
\item Suppose \(H\leq L\leq G\) and \(\lambda\) extends to
\(\widetilde\lambda\in\Lin(L)\).  If
\(\Ind_H^G\lambda\) is MF, then \((L,H)\) is a Gelfand pair.
\end{enumerate}
\end{proposition}
\begin{proof}
Parts (i)--(ii) are the standard commutant and Mackey descriptions of
finite Gelfand triples \cite{CSST}.
For (iii), inversion fixes \(e_\lambda\) and every basis element
\(e_\lambda g e_\lambda\) indexed by an admissible double coset.
It is an anti-automorphism, so the corner algebra is commutative; this
is the twisted Gelfand trick in \cite[Proposition III.1.2]{RSW}.

For (iv),
\[
 \Ind_H^L\lambda
 =\widetilde\lambda\otimes\C[L/H].
\]
If \(\C[L/H]\) had a constituent of multiplicity at least two, the
same would hold after twisting and then inducing to \(G\), since
induction takes a nonzero genuine character to a nonzero genuine
character.  This contradicts the MF hypothesis.
\end{proof}

Consequently the classification below has two separate layers.
Every allowed term first has to define a linear Gelfand triple, and its
ambient pair \((W_J,H)\) has to be a Gelfand pair by part (iv).
The supports of the surviving triples must then form an exact cover.
Commutativity of the individual corner algebras proves the first
condition but does not imply the second.

\begin{proposition}[Corank obstruction {\cite[Theorem 1.6]{MZperfect}}]
\label{prop:corank}
Suppose that \(W\) is irreducible and \(|S\setminus J|\geq2\).
Then \(\Ind_{W_J}^W\rho\) is not MF for any \(\rho\in\Irr(W_J)\).
Consequently every allowed model term has \(J=S\) or \(|S\setminus J|=1\).
Also, a subgroup contained in a parabolic subgroup of corank at least two cannot contribute an MF induced character, for any inducing character.
\end{proposition}

The consequences follow by induction in stages and positivity.

For a subgroup \(K\leq G\), write
\(N_G(K)=\{g\in G:gKg^{-1}=K\}\), and write
\(C_G(x)=\{g\in G:gx=xg\}\) for the centralizer of an element.

\begin{proposition}[Finite reduction]\label{prop:finite}
For each \(J=S\) or \(|S\setminus J|=1\), enumerate the subgroups of \(W_J\) up to conjugacy.
Retain only \(H\) with \([W:H]\leq\Gamma_W(1)\) for which \((W_J,H)\) is a Gelfand pair, and compute
\[
 m_\rho(J,H,\sigma)
 =\frac1{|H|}\sum_{h\in H}\sigma(h)\overline{\rho(h)},
 \qquad\rho\in\Irr(W),\quad\sigma\in\Lin(W_J).
\]
Discard vectors with an entry greater than one and identify equal vectors.
Their exact covers give all models when the QP condition is omitted.
For QP models, retain a vector precisely when at least one of its realizations has a coset action admitting a QP minimum-length height based at some point.
\end{proposition}
\begin{proof}
Proposition~\ref{prop:corank} makes the list of parabolic sources exhaustive.
Every model term is MF and has degree at most \(\Gamma_W(1)\).
Since \(\sigma\) extends to \(W_J\), Proposition~\ref{prop:Gelfandcriteria}(iv) justifies the preliminary Gelfand-pair test.
Frobenius reciprocity gives the displayed entries.
Conjugating \(H\) within \(W_J\) preserves all these induced characters.

For the additional QP test, fix a point \(b\) of each coset action and assign to \(x\) its shortest simple-generator distance from \(b\).
Check QP1 and QP2 using all reflections of \(W_J\).
If the test succeeds, \(\Stab_{W_J}(b)\) realizes the required character and is QP.
Conversely, Proposition~\ref{prop:QPfacts}(i) says that the height of every finite transitive QP set is, up to translation, obtained this way from its minimum.
Testing all possible \(b\)'s is therefore exhaustive.
Finally, nonnegative integral multiplicities make the model equation exactly the stated exact-cover problem.
\end{proof}

The archive \cite{ZhangData} distinguishes three stages: exhaustive subgroup generation, calculation of exact multiplicity vectors, and verification of the exported data.
The ordinary GAP sources use the subgroup algorithms and the complete tables of marks in \cite{GAP,CTblLib,TomLib}; a table of marks records, up to subgroup conjugacy, the numbers of fixed cosets in transitive permutation actions.
Independent readers check character orthogonality, degrees, all accepted QP actions, a violating axiom at every rejected basepoint, and the exact covers.
These readers verify the exported finite records; completeness of the subgroup input comes from the specified enumeration and the reductions just proved.
The archive fixes the irreducible row ordering by the full character values, so equal degrees never identify different irreducible characters.

More generally, if \(\Ind_H^L\tau\) contains an irreducible character twice, its induction to any overgroup still has a repeated constituent.
If \(\tau\) extends from \(H\) to \(L\), then \(\Ind_H^G\tau\) contains \(\Ind_L^G\tau\) as a positive subcharacter.
We will repeatedly use these two observations.

A partition \(\lambda\vdash n\) is a weakly decreasing sequence of nonnegative integers with sum \(n\); zeros are suppressed.
Write \(p(n)\) for the number of partitions of \(n\), with \(p(0)=1\).
Write \(\chi^\lambda\) for the irreducible character of \(S_n\) indexed by \(\lambda\).
Its Young diagram has \(\lambda_i\) cells in row \(i\), in English notation.
Its transpose is \(\lambda^\transpose\).
A standard Young tableau fills its diagram with \(1,\ldots,n\), increasingly along rows and columns; the number of such tableaux is \(f^\lambda=\chi^\lambda(1)\).
The empty partition is \(\emptypar\).
Let
\[
 o_r(\lambda)=\#\{i:\lambda_i\text{ is odd}\},
 \qquad o_c(\lambda)=o_r(\lambda^\transpose).
\]
An even-row partition has all parts even.
For an arbitrary partition \(\lambda\), define its \emph{row-even remainder} by
\[
 \lambda^{\mathrm{ev}}=(2\lfloor\lambda_1/2\rfloor,2\lfloor\lambda_2/2\rfloor,\ldots).
\]
This deletes the final cell of each odd row.

The Frobenius characteristic takes \(\chi^\lambda\) to the Schur function \(s_\lambda\), trivial characters to \(h_n=s_{(n)}\), and signs to \(e_n=s_{(1^n)}\).
Induction from \(S_a\times S_b\) corresponds to multiplication.
The Littlewood--Richardson coefficients \(c^\nu_{\lambda\mu}\) are defined by
\(s_\lambda s_\mu=\sum_\nu c^\nu_{\lambda\mu}s_\nu\).
Equivalently, they count semistandard tableaux of skew shape \(\nu/\lambda\), content \(\mu\), whose word read from right to left along successive rows, starting at the top, is a lattice word.
Here semistandard means weakly increasing rows and strictly increasing columns; a lattice word has at least as many \(i\)'s as \(i+1\)'s in every initial segment.
The Pieri rules assert that multiplication by \(h_q\), respectively \(e_q\), adds \(q\) cells with no two in the same column, respectively row.
We use these rules, the Cauchy identity, and the Murnaghan--Nakayama rule as in \cite{MacdonaldBook}.
The latter computes \(\chi^\lambda(\mu)\) by successively removing rim hooks of the cycle lengths of \(\mu\), with sign \((-1)^{\text{number of rows}-1}\) for each rim hook.
Multiplicity-free Schur products are treated systematically in \cite{StembridgeSchur}.

\subsection{The classical QP subgroup input}
\label{subsec:Lu}

Write \(B_n=N_n\rtimes S_n\), where \(N_n=\langle f_1,\ldots,f_n\rangle\cong C_2^n\), and \(f_i\) changes the sign of coordinate \(i\).
The simple generators are \(s_0=f_1\) and \(s_i=(i,i+1)\) for \(1\leq i<n\).
Let
\[
 p:B_n\to S_n,\quad
 \ee(w)=(-1)^{\#\{i:w(i)<0\}},\quad
 \dd(w)=\sgn(p(w)),\quad \om=\ee\dd .
\]
Then \(\om=\sgn_{B_n}\).
Put \(N_n^+=\Ker(\ee|_{N_n})\) and \(D_n=N_n^+\rtimes S_n\).
The two fork generators of \(D_n\) are \(t=f_1f_2(12)\) and \(u=(12)\); the other simple generators are \((23),\ldots,(n-1,n)\).
Conjugation by \(f_1\) defines the diagram automorphism \(\diagauto\) of \(D_n\).

On \(2k\) consecutive coordinates define
\[
 \pair_k=C_2\wr S_k\leq S_{2k},\qquad
 a_i=(2i-1,2i),\quad c_i=f_{2i-1}f_{2i},
\]
where \(\pair_k\) preserves the indicated pairing.
Let \(b_i=(2i-1,2i+1)(2i,2i+2)\) interchange neighboring pairs.
The signed pairing subgroup and its even-special subgroup are
\begin{equation}
 C_k=\langle a_i,c_i,b_i\rangle\cong(C_2\times C_2)\wr S_k,\qquad
 E_k=\Ker\eta_k,\quad
 \eta_k(a_i)=\eta_k(b_i)=1,\quad\eta_k(c_i)=-1 .
 \label{eq:CE}
\end{equation}
We put \(C_0=1\).
Lu's \emph{A form} means an unsigned QP subgroup.
On a signed initial segment, the \emph{B form} has full sign kernel \(N_d\);
the two \emph{D forms} available in \(B_d\) are
\(N_d^+\rtimes L\) and \((N_d\rtimes L)^\circ\).
In \(D_d\) only the first applies.
The special and even-special forms are \(C_k\) and \(E_k\), with the permitted diagram images.
The signed forms occur on an initial segment; their remaining unsigned orbits have a uniform choice of sign.

A fiber product of surjections \(q_i:L_i\to Q\) is
\(\{(x_1,x_2):q_1(x_1)=q_2(x_2)\}\).
Goursat's lemma says that every subgroup \(H\leq L_1\times L_2\) projecting onto both factors has this form.
Indeed, put \(K_1=\{x:(x,1)\in H\}\) and \(K_2=\{y:(1,y)\in H\}\).
Then \(K_i\) is normal in \(L_i\), and membership in \(H\) defines an isomorphism \(L_1/K_1\cong L_2/K_2\), giving the assertion.
In the following list, \(V_4\) is the normal Klein four subgroup of \(S_4\).

\begin{proposition}[Classical subgroup input extracted from Lu]
\label{prop:Lu}
The following lists and restrictions are consequences of
\cite[Theorems 11, 13, 16; Definitions 9--13; Examples 20--21]{Lu}.
\begin{enumerate}[label=(\roman*)]
\item QP subgroups of \(S_n\) have consecutive natural orbits.
They are products on consecutive supports of symmetric groups; pairing groups \(\pair_k\); total even subgroups of products of symmetric and pairing groups; diagonal symmetric groups; the two fiber products described below; and four transitive exceptional factors in degrees \(5,6,7,8\).
\item The four exceptional factors are a dihedral group of order \(10\) on five points, the transitive \(A_5\) on six points, \(\GL_3(2)\) on seven points, and \(\AGL_3(2)\) on eight points.
The fiber products have orbit sizes \(4+3\) or \(4+4\): they are
\[
 \{(x,y)\in S_4\times S_3:\bar x=y\},\qquad
 \{(x,y)\in S_4\times S_4:\bar x=\bar y\},
 \quad S_4/V_4\cong S_3 .
\]
The diagonal factor has two equal orbits with the same permutation action.
\item For a single pairing image \(\pair_k\), the QP lifts in \(D_{2k}\) have A, D, special, or even-special form.
The lifts in \(B_{2k}\) have A form, either D form, B form, or special form.
The even-special form is not QP in \(B_{2k}\).
Special signed pairings start at coordinate \(1\) and pair consecutive coordinates.
There is in addition a signed order-\(16\) exceptional lift in rank \(4\).
\item For the single image \(S_d\), in ranks \(d\geq5\) the possible lifts are \(S_d\), \(N_d^+\rtimes S_d\), and, in type \(B\), \(N_d\rtimes S_d\) and its Coxeter-even subgroup when the image is retained.
The additional exceptional signed cores have ranks at most \(4\) for this image.
\item Couplings between the orbit factors are obtained by the simple double-cover constructions in \cite[Theorem 16]{Lu}.
Thus they must be retained when intersecting with a maximal parabolic subgroup; no direct-product assumption on a subgroup of that parabolic is imposed.
\end{enumerate}
\end{proposition}

Items (i)--(ii) are the orbit and factor lists of Lu's Theorems 11 and 13.
Items (iii)--(v) are the restrictions of Theorem 16 to the specified unsigned images, rather than additional classification theorems quoted from Lu.
To obtain them, retain only factors with that image and apply the inductive, projective, and simple double-cover operations in that theorem.
The signed exceptional factors in ranks six, seven, and eight have respectively the transitive exceptional images in those degrees, so they do not survive a single symmetric or pairing image requirement; the rank-four exceptional pairing lift is retained explicitly.
The graph-extension and small-kernel consequences needed later are proved in Lemmas~\ref{lem:graphextend}, \ref{lem:Bsmallkernel}, and \ref{lem:Dprojection}.

For clarity, a double-cover coupling of two even subgroups \(L_i\) with reflection extensions
\(\widetilde L_i=\langle L_i,r_i\rangle\), \([\widetilde L_i:L_i]=2\), is the fiber product
\(\{(x_1,x_2):x_iL_i\text{ have the same parity}\}\).
It is simple when the reflection extensions can be represented by the simple reflections specified in Lu's construction.

\subsection{The character formulas used in the classical cases}
\label{subsec:charformulas}

The irreducible characters of \(B_n\) are indexed by ordered bipartitions
\((\alpha,\beta)\) with \(|\alpha|+|\beta|=n\).
Their degrees are
\(\binom n{|\alpha|}f^\alpha f^\beta\), and
\[
 \ee\chi^{(\alpha,\beta)}=\chi^{(\beta,\alpha)},\qquad
 \dd\chi^{(\alpha,\beta)}=\chi^{(\alpha^\transpose,\beta^\transpose)}.
\]
The irreducibles \(\chi^{(\lambda,\emptypar)}\) are the inflations from \(S_n\).
Induction between signed Young subgroups multiplies the two Schur-function variables separately.
For induction of an unsigned \(S_q\)-factor, the coefficient at \((\alpha,\beta)\) is \(c^\nu_{\alpha\beta}\) when that factor has character \(\chi^\nu\).

For \(D_n\), restriction of \(\chi^{(\alpha,\beta)}\), \(\alpha\ne\beta\), is the irreducible
\(\chi^{\{\alpha,\beta\}}\).
When \(n=2m\), restriction of \(\chi^{(\lambda,\lambda)}\) is
\(\chi^{[\lambda,+]}+\chi^{[\lambda,-]}\); these are the \emph{degenerate associates}.
The diagram automorphism exchanges the two associates and fixes each nondegenerate character.
We fix their signs by the convention of \cite[Section 5.1]{MZperfect}:
if \(z_m=a_1\cdots a_m\), then
\begin{equation}
 \chi^{[\lambda,+]}(z_m)-\chi^{[\lambda,-]}(z_m)=2^m f^\lambda .
 \label{eq:associate-sign}
\end{equation}
The type \(D\) Littlewood--Richardson rule, including this sign information, is given in \cite{Taylor} and \cite[(5.8)]{MZperfect}.

Let \(E^{A}_{2k}=\sum_{\lambda\vdash2k,\ o_r(\lambda)=0}s_\lambda\).
The pairing formula and Pieri rule give
\begin{equation}
 \ch\Ind_{\pair_k}^{S_{2k}}\one=E^{A}_{2k},\qquad
 E^{A}_{2k}e_q=\sum_{o_r(\lambda)=q}s_\lambda .
 \label{eq:Apieri}
\end{equation}
Transposition gives the column version.
For \(q=n-2k\), define
\begin{equation}
 \Theta_k=\Ind_{C_k\times S_q}^{B_n}(\one\boxtimes\sgn),\qquad
 \Phi_k=\Ind_{C_k\times S_q}^{D_n}(\one\boxtimes\sgn).
 \label{eq:ThetaPhi}
\end{equation}
For \(\Phi_0\), the source is the standard unsigned \(S_n\).
The following formulas are the character computations of \cite[Sections 4.2 and 5.3]{MZperfect}:
\begin{align}
 \Theta_k&=\sum_{o_r(\alpha)+o_r(\beta)=q}\chi^{(\alpha,\beta)},\label{eq:ThetaSupport}\\
 \Phi_k&=\sum_{\substack{\alpha\ne\beta\\o_r(\alpha)+o_r(\beta)=q}}
             \chi^{\{\alpha,\beta\}}
       +\sum_{\substack{\lambda\vdash n/2\\2o_r(\lambda)=q}}\chi^{[\lambda,\epsilon_{\lambda,k}]} .
 \label{eq:PhiSupport}
\end{align}
The first sum in \eqref{eq:PhiSupport} is over unordered pairs; the second is zero for odd \(n\).
The sign \(\epsilon_{\lambda,k}\) depends on the embedding and is computed by the cited type \(D\) rule.
For the fixed consecutive embedding it is \((-1)^{n/2}\) in the row family.
In particular,
\begin{equation}
 \Ind_{C_m}^{D_{2m}}\one
 =\sum_{\substack{\alpha\ne\beta\\\alpha,\beta\ {\rm even\!-\!row}}}
        \chi^{\{\alpha,\beta\}}
  +\sum_{\substack{\lambda\vdash m\\\lambda\ {\rm even\!-\!row}}}\chi^{[\lambda,+]} .
 \label{eq:Ctop}
\end{equation}
One way to verify the uniform sign in the row family is to apply the signed Littlewood--Richardson rule to its unique even-row predecessor: a vertical strip of size \(o_r(\lambda)\) contributes the same sign \((-1)^{n/2}\).
These formulas concern ordinary special subgroups; the character of \(E_m\) will be proved in Lemma~\ref{lem:half}.

\subsection{Hecke modules, bar operators, and \texorpdfstring{\(W\)}{W}-graphs}
\label{subsec:Heckedef}

Set \(\A=\Z[v,v^{-1}]\) and \(a=v-v^{-1}\).
The equal-parameter Hecke algebra \(\HH(W)\) is generated by \(T_s\), \(s\in S\), subject to the braid relations and
\[
 T_s^2=1+aT_s,\qquad
 \overline v=v^{-1},\qquad \overline{T_s}=T_s^{-1}=T_s-a.
\]
The displayed assignments define its bar involution.
A \emph{bar operator} on a free \(\A\)-module \(M\) is an \(\A\)-antilinear involution \(\psi\) such that
\(\psi(hm)=\bar h\,\psi(m)\).
An ordered standard basis \(\{m_x:x\in X\}\), together with
\(\psi(m_x)\in m_x+\sum_{y<x}\A m_y\), is a triangular bar structure.
Its \emph{canonical basis} is the basis satisfying
\begin{equation}
 \psi(C_x)=C_x,\qquad
 C_x\in m_x+\sum_{y<x}v^{-1}\Z[v^{-1}]m_y .
 \label{eq:canonical-def}
\end{equation}
Such a basis, when it exists, is unique.

An \emph{integral \(W\)-graph}, in the convention used in this paper, is a finite set \(X\), labels
\(\tau(x)\subseteq S\), and integer weights \(\omega(x,y)\), for which
\begin{equation}
 T_s C_x=
 \begin{cases}
 vC_x,&s\notin\tau(x),\\
 -v^{-1}C_x+\displaystyle\sum_{s\notin\tau(y)}\omega(x,y)C_y,&s\in\tau(x)
 \end{cases}
 \label{eq:Wgraph-def}
\end{equation}
satisfies the Hecke relations.
Weights over a splitting field define a \(W\)-graph over that field.
We do not impose positivity of the weights.
A \emph{canonical \(W\)-graph} is one obtained from a canonical basis in this sense.
Specialization means setting \(v=1\); preserving a natural lattice means that this specialization is isomorphic over \(\Z W\) to the specified induction lattice.

\begin{proposition}[Ordinary QP Hecke theory]
\label{prop:barinput}
For a finite transitive QP set \(X\), the following formulas define the Rains--Vazirani module \(M(X)\):
\begin{equation}
 T_s m_x=
 \begin{cases}
 m_{sx},&h(sx)>h(x),\\
 m_{sx}+a m_x,&h(sx)<h(x),\\
 vm_x,&sx=x.
 \end{cases}
 \label{eq:ordinaryM}
\end{equation}
The module \(N(X)\) has fixed-point scalar \(-v^{-1}\).
For finite classical \(W\), both modules have bar operators fixing their minimum vectors and triangular with respect to the QP Bruhat order.
They have canonical bases and integral \(W\)-graphs.
\end{proposition}

The module construction is \cite[Theorem 7.1]{RV}, the classical bar theorem is \cite[Theorem 18]{Lu}, and the canonical and graph statements are \cite[Theorems 3.14, 3.16, 3.26]{MarbergBar}.
The specialization of \(M(W/H)\) is \(\Z[W/H]\); that of \(N(W/H)\) is isomorphic to \(\Ind_H^W\Z_{\sgn_W|_H}\).
For \(N\), the isomorphism includes the usual height-sign change in a coset basis.

A \emph{perfect involution} is an element \(z\) of the semidirect product of a Coxeter group with its diagram automorphisms such that \(z^2=1\) and \((zt)^4=1\) for every reflection \(t\).
The relevant minimum centralizers are QP \cite{RV}.
A \emph{perfect model} uses these centralizers in its inducing terms.
The constructions of \cite{MZgraphs}, together with parabolic \(W\)-graph induction \cite{HY1,HY2}, give natural bar and canonical \(W\)-graph realizations of the perfect families used below.
On an induced module the bar is
\(\overline{h\otimes m}=\bar h\otimes\bar m\).

\begin{proposition}[Character-level existence {\cite{Gyoja,Hahn}}]
\label{prop:Gyoja}
For finite \(W\), every irreducible generic Hecke module over a splitting field has a \(W\)-graph over that field.
The Tits deformation correspondence identifies these modules with \(\Irr(W)\).
\end{proposition}

Here the deformation correspondence is used in the form of \cite{GeckPfeiffer}.
In a \(W\)-graph basis, coefficient conjugation \(v\mapsto v^{-1}\), fixing the vertices, is a compatible bar operator.
This follows directly from \eqref{eq:Wgraph-def}.
\section{Type A Coxeter group}
\label{sec:A}

For \(0\leq q\leq n\), \(q\equiv n\pmod2\), set
\[
 R_q^A=\sum_{\lambda\vdash n,\ o_r(\lambda)=q}\chi^\lambda,
 \qquad C_q^A=\sum_{\lambda\vdash n,\ o_c(\lambda)=q}\chi^\lambda.
\]
The superscript will be suppressed when only symmetric groups occur.
These are the row and column characters.
By \eqref{eq:Apieri}, for \(n=2k+q\),
\[
 R_q=\Ind_{\pair_k\times S_q}^{S_n}(\one\boxtimes\sgn),
 \qquad C_q=\sgn_{S_n}R_q .
\]
They are allowed characters, and each of the two families is a model.

\begin{theorem}\label{thm:A}
If \(n=3\) or \(n\geq5\), then
\(\QQ(S_n)=\{\{R_q\}_q,\{C_q\}_q\}\).
For \(S_2\), the two models are \(\{\one,\sgn\}\) and \(\{\one+\sgn\}\).
For \(S_4\), there are four models: the row and column families, and
\[
 \{\chi^{(4)}+\chi^{(3,1)}+\chi^{(2,2)},\
       \chi^{(2,1,1)}+\chi^{(1^4)}\}
\]
together with its transpose.
\end{theorem}

\begin{corollary}[Gelfand triples in type A]
\label{cor:AGelfand}
Let \(n=2k+q\), where \(q\equiv n\pmod2\), and put
\[
 H_{k,q}=\pair_k\times S_q,\qquad
 \lambda_{k,q}^{\row}=\one\boxtimes\sgn_{S_q},\qquad
 \lambda_{k,q}^{\col}=
   (\sgn_{S_{2k}}|_{\pair_k})\boxtimes\one .
\]
Then \((S_n,H_{k,q},\lambda_{k,q}^{\row})\) and
\((S_n,H_{k,q},\lambda_{k,q}^{\col})\) are linear Gelfand triples.
If
\[
 g_{n,q}=\#\{\lambda\vdash n:o_r(\lambda)=q\},
\]
then their twisted double-coset algebras are both isomorphic to
\(\C^{g_{n,q}}\).
Moreover,
\[
 (S_{2k}\times S_q,H_{k,q})
\]
is an ordinary Gelfand pair whose double-coset algebra is
\(\C^{p(k)}\).
Every inducing term in the two exceptional \(S_4\) models of
Theorem~\ref{thm:A} is likewise a linear Gelfand triple.
\end{corollary}
\begin{proof}
The two induced characters are \(R_q\) and \(C_q\), respectively.
Their displayed decompositions are MF and contain exactly \(g_{n,q}\)
constituents, so Proposition~\ref{prop:Gelfandcriteria}(i) gives the
first two assertions.
For the ordinary pair, its permutation character is
\[
 \Ind_{\pair_k}^{S_{2k}}\one\boxtimes\one
 =\sum_{\mu\vdash k}\chi^{2\mu}\boxtimes\one,
\]
which has \(p(k)\) distinct constituents.
The last assertion follows from the two MF decompositions in
Theorem~\ref{thm:A}.
\end{proof}

\begin{remark}[Comparison with the Gelfand-pair literature]
The pair \((S_{2k},\pair_k)\) and its zonal spherical functions are
classical \cite[Chapter VII]{MacdonaldBook}.
The triples above are the individual layers of the
Inglis--Richardson--Saxl model \cite{IRS}; thus their MF property is not
claimed here as new.
The ordinary permutation case is also covered by the complete
classification of Godsil and Meagher \cite{GodsilMeagher}, and the
large-degree classification of general MF induced characters in
\cite{Turek} gives a further comparison.
The new assertion of Theorem~\ref{thm:A} is that, after imposing QP
stabilizers, ambient extension, and exact-cover equivalence, these
layers admit no other assembly except in the stated low ranks.
\end{remark}

\subsection{A complete candidate reduction}

Every model term is MF.
By Proposition~\ref{prop:Gelfandcriteria}(iv), its source subgroup is
first constrained by an ordinary Gelfand-pair condition inside its
ambient parabolic subgroup.
The natural orbits of a subgroup of \(S_n\) determine its parabolic closure.
Proposition~\ref{prop:corank} therefore permits at most two natural orbits.
Apply the entire list in Proposition~\ref{prop:Lu}.
The possible sources are one or two transitive factors, total even subgroups of two symmetric or pairing factors, diagonal symmetric groups, and the two fiber products.
Independent ambient signs are retained on the two orbits.

\begin{lemma}\label{lem:sporadic}
For the four transitive exceptional factors of degrees \(r=5,6,7,8\), the permutation Frobenius characteristic is
\begin{equation}
 F_r=s_{(r)}+s_{(a,b)}+s_{(a,b)^\transpose}+s_{(1^r)},
 \quad (r;a,b)=(5;3,2),(6;3,3),(7;4,3),(8;4,4).
 \label{eq:sporadic}
\end{equation}
If \(q\geq b\), neither \(F_rh_q\) nor \(F_re_q\) is MF.
Consequently an MF candidate containing an exceptional factor and a symmetric or alternating tail has \(n\leq11\).
\end{lemma}
\begin{proof}
For \eqref{eq:sporadic}, average the irreducible symmetric-group characters over the four explicit permutation groups of orders \(10,60,168,1344\).
The Murnaghan--Nakayama rule gives the four displayed supports with coefficient one.

For the bound, let \(\nu=(r+q-b,b)\).
It occurs in both \(s_{(r)}h_q\) and \(s_{(a,b)}h_q\): the first assertion uses \(q\geq b\), and the second follows from the horizontal-strip interlacing inequalities
\(r+q-b\geq a\geq b\).
Thus its multiplicity in \(F_rh_q\) is at least two.
Since \(F_r\) is invariant under transposition, the same holds for \(F_re_q\).
Hence \(n=r+q\leq r+b-1\leq11\).
\end{proof}

If the other factor is a pairing group on at least four points, the single-row or single-column summand in \(F_r\), with the appropriate orientation, gives a repeated constituent by \cite[Lemma 3.1(a,b)]{MZperfect}.
Two exceptional factors are excluded by Lemma~\ref{lem:sporadic}, since the second has a single-row summand of degree at least five.
Their alternating nature introduces no further total-even case.

For \(n\geq12\), call a character \emph{thin} if every partition in its support avoids the diagram \((3,2,1)\).
The diagonal candidates have Frobenius characteristics
\[
 \sum_{\lambda\vdash i}s_\lambda s_\lambda,\qquad
 \sum_{\lambda\vdash i}s_\lambda s_{\lambda^\transpose}.
\]
For \(i\geq2\) the first has a repeated \((2i-2,2)\), from \(\lambda=(i)\) and \((i-1,1)\).
The second has the same nonzero product from \(\lambda=(i)\) and \((1^i)\).
The two fiber products have degrees \(7\) or \(8\), already use two orbits, and cannot occur in this range.
Products involving only symmetric and alternating groups are positive sums of
\(h_ah_b,e_ae_b,h_ae_b,e_ah_b\), all of which are thin.

For \(k,q\geq2\), the same-oriented product \(E^A_{2k}h_q\) has a repeated \((n-2,2)\), from \((2k)\) and \((2k-2,2)\).
Products of two genuine pairing characters, in any orientation, are not MF by \cite[Lemma 3.1(c)]{MZperfect}.
An even-pairing source is handled by its two positive extensions.
The index-two formula now proves the following reduction.

\begin{lemma}\label{lem:Acolumns}
For \(n\geq12\), every MF allowed character is thin, a single \(R_q\) or \(C_q\), or
\[
 R_q+C_q,\qquad \Supp(R_q)\cap\Supp(C_q)=\emptypar .
\]
For \(n\geq5\), an allowed term whose character is a nontrivial row or column layer \(R_q\) or \(C_q\) has, up to its permitted embeddings, subgroup
\(\pair_k\times S_q\), \(2k+q=n\).
The layer \(R_n\) comes from \(S_n\).
\end{lemma}
\begin{proof}
The preceding exclusions leave exactly the pairing character with the opposite-oriented symmetric tail and its index-two total-even construction.
These give \(R_q,C_q\), or their sum.
At \(q=1\), the sum is not MF because \((k+1,k)\) belongs to both supports.
The case \(k=1\) is already among the symmetric factors.
This proves the first assertion.
The subgroup realization assertion follows by the same retained-factor list for \(n\geq12\).
For \(5\leq n\leq11\), retain the subgroup provenance of each column in the finite catalogue described below; every realization of these layers has exactly the stated factors and permitted embeddings \cite{ZhangData}.
\end{proof}

\subsection{Connected support equations}

\begin{lemma}\label{lem:Aconnected}
For \(n\geq12\), form a bipartite graph with vertices
\[
 r_q,c_q\qquad(0\leq q\leq n-4,\ q\equiv n\pmod2).
\]
Join \(r_q\) to \(c_p\) if some partition containing \((3,2,1)\) has \((o_r,o_c)=(q,p)\).
This graph is connected.
\end{lemma}
\begin{proof}
For odd \(n\), the following partitions have \((o_r,o_c)=(1,p)\):
\[
 \lambda_p=
 \begin{cases}
 ((n+p-2)/2,(n-p)/2,1),&p\equiv n\pmod4,\\
 ((n+p-4)/2,(n-p)/2,2),&p\not\equiv n\pmod4,\ p\geq3,\\
 ((n-3)/2,(n-3)/2,2,1),&p=1,\ n\equiv3\pmod4.
 \end{cases}
\]
All contain \((3,2,1)\).
Their transposes connect every vertex to \(r_1,c_1\).
For even \(n\) and \(p>0\), use
\(((n+p-2)/2,(n-p)/2,1)\), whose statistics are \((2,p)\).
If \(n\equiv2\pmod4\), the partition
\(((n-2)/2,(n-2)/2,1,1)\) connects \(r_2\) to \(c_0\).
If \(4\mid n\), use \((n-4,2,2)\) to connect \(r_0\) to \(c_{n-4}\).
Transposition completes the connections in both cases.
\end{proof}

\begin{proof}[Proof of Theorem~\ref{thm:A} for \(n\geq12\)]
For each selected term in Lemma~\ref{lem:Acolumns}, record its row and column layers.
Write \(x_q,y_p\in\{0,1\}\) for the respective selection indicators.
A nonthin partition on an edge of Lemma~\ref{lem:Aconnected} gives
\[
 x_q+y_p=1.
\]
Connectivity forces either all \(x_q=1,y_p=0\), or all \(x_q=0,y_p=1\).
Thus the nonthin layers have a common orientation and no combined \(R_q+C_q\) occurs there.

Take the row orientation.
The layers with \(q\leq n-4\) leave only
\[
 (3,1^{n-3}),\qquad(2,1^{n-2}),\qquad(1^n).
\]
In a thin candidate, the first can only be supplied by \(h_2e_{n-2}\) or \(h_3e_{n-3}\).
The latter also supplies \((4,1^{n-4})\), already covered.
The former is exactly \(R_{n-2}\).
Any additional positive extension from an alternating or total-even source supplies an already covered two-row, two-column, or opposite hook partition.
It therefore cannot merely append the remaining sign.
The only completion is \(R_{n-2},R_n\).
Transposition gives the column completion.
\end{proof}

\subsection{Low ranks and examples}

For \(n\leq11\), the corank obstruction leaves at most two natural orbits.
Thus Proposition~\ref{prop:Lu} gives a finite catalogue: the one-orbit factors, their two-factor products and total-even couplings, the diagonal groups, and the two fiber products.
All restrictions of signs of the ambient Young factors are included.
Induction in stages, the index-two formula, and the Schur rules of Section~\ref{sec:prelim} compute their characters; direct averaging supplies the four transitive exceptional factors.
The number of distinct MF columns for \(n=2,\ldots,11\) is
\[
 3,5,11,16,26,29,41,43,54,57,
\]
and their exact-cover counts are respectively
\[
 2,2,4,2,2,2,2,2,2,2.
\]
The column supports and the retained subgroup provenance identify precisely the families in Theorem~\ref{thm:A} and verify the low-rank assertion of Lemma~\ref{lem:Acolumns}.
The catalogue generator and exact reader are supplied in \cite{ZhangData}; for \(5\leq n\leq8\), a separate all-subgroup GAP enumeration gives the same entire MF QP column set.
This completes the finite part of the proof.

\begin{example}[A row model of \(S_5\)]
The three terms have supports
\[
 R_1=\chi^{(5)}+\chi^{(4,1)}+\chi^{(3,2)}+\chi^{(2,2,1)},\quad
 R_3=\chi^{(3,1,1)}+\chi^{(2,1,1,1)},\quad
 R_5=\chi^{(1^5)}.
\]
Their degrees are \(15,10,1\), adding to \(26=\Gamma_{S_5}(1)\).
For \(R_1=E^A_4e_1\), the even-row predecessors are \((4)\) and \((2,2)\).
The second predecessor gives the last two diagrams shown in Figure~\ref{fig:pieri}.
\end{example}

\begin{figure}[htbp]
\centering
\begin{tikzpicture}[x=0.42cm,y=-0.42cm]
\foreach \x in {0,1,2,3}{\draw (\x,0) rectangle +(1,1);}
\node at (2,2) {\(\lambda=(4)\)};
\begin{scope}[xshift=3.2cm]
\foreach \y in {0,1}{\foreach \x in {0,1}{\draw (\x,\y) rectangle +(1,1);}}
\node at (1,3) {\(\lambda=(2,2)\)};
\end{scope}
\begin{scope}[xshift=6.5cm]
\foreach \y in {0,1}{\foreach \x in {0,1}{\draw (\x,\y) rectangle +(1,1);}}
\filldraw[fill=blue!18] (2,0) rectangle +(1,1);
\node at (1.5,3) {\(\nu=(3,2)\)};
\end{scope}
\begin{scope}[xshift=9.6cm]
\foreach \y in {0,1}{\foreach \x in {0,1}{\draw (\x,\y) rectangle +(1,1);}}
\filldraw[fill=blue!18] (0,2) rectangle +(1,1);
\node at (1,4) {\(\nu=(2,2,1)\)};
\end{scope}
\end{tikzpicture}
\caption{Even-row predecessors and two outputs of the vertical Pieri rule. Shaded cells are added cells.}
\label{fig:pieri}
\end{figure}
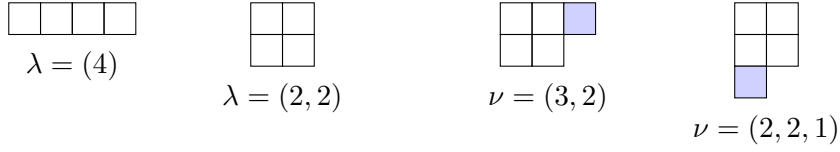

\begin{example}[The extra \(S_4\) family]
The two characters outside the row/column families may be written
\[
 \Ind_{S_2\times S_2}^{S_4}\one
   =\chi^{(4)}+\chi^{(3,1)}+\chi^{(2,2)},\qquad
 \Ind_{S_3}^{S_4}\sgn
   =\chi^{(2,1,1)}+\chi^{(1^4)}.
\]
Their degrees are \(6\) and \(4\).
The three standard tableaux of shape \((3,1)\) illustrate the degree-three constituent:
\[
 \begin{array}{|c|c|c|}\hline1&2&3\\\hline4\\\cline{1-1}\end{array}
 \qquad
 \begin{array}{|c|c|c|}\hline1&2&4\\\hline3\\\cline{1-1}\end{array}
 \qquad
 \begin{array}{|c|c|c|}\hline1&3&4\\\hline2\\\cline{1-1}\end{array}.
\]
\end{example}
\section{Type B Coxeter group}
\label{sec:B}

Retain \(\ee,\dd,\om,C_k,\Theta_k\) from Section~\ref{sec:prelim}, and define
\[
 U_n=\Ind_{B_2\times S_{n-2}}^{B_n}(\one\boxtimes\sgn).
\]
The index-two formula in the first factor gives
\begin{equation}\Theta_1=U_n+\ee U_n.\label{eq:Bsplit}\end{equation}
For an unordered family \(\mathcal M\), write
\(\om\mathcal M=\{\om\chi:\chi\in\mathcal M\}\).

\begin{theorem}\label{thm:B}
For every \(n\geq4\), the four elements of \(\QQ(B_n)\) are
\[
 \mathcal B_n=\{\Theta_k:0\leq k\leq\lfloor n/2\rfloor\},
 \qquad
 \mathcal B_n'=(\mathcal B_n\setminus\{\Theta_1\})\cup\{U_n,\ee U_n\},
 \qquad \om\mathcal B_n,\quad\om\mathcal B_n'.
\]
\end{theorem}
Their existence and perfect realizations are given in \cite[Theorem 4.5]{MZperfect}.
We prove that arbitrary QP stabilizers yield no additional families.

\begin{corollary}[Gelfand triples in type B]
\label{cor:BGelfand}
Let \(n=2k+q\), put \(H_{k,q}=C_k\times S_q\), and let
\(\lambda_{k,q}=\one\boxtimes\sgn_{S_q}\).
Then \((B_n,H_{k,q},\lambda_{k,q})\) is a linear Gelfand triple.
If
\[
 b_{n,q}=\#\{(\alpha,\beta):|\alpha|+|\beta|=n,
          \ o_r(\alpha)+o_r(\beta)=q\},
\]
then
\[
 \mathscr A(B_n,H_{k,q},\lambda_{k,q})\cong\C^{b_{n,q}}.
\]
In particular, \((B_{2k},C_k)\) is an ordinary Gelfand pair and
\[
 \mathscr A(B_{2k},C_k,\one)
 \cong\C^{\sum_{i=0}^k p(i)p(k-i)}.
\]
The two characters \(U_n\) and \(\ee U_n\) in the split family also
define linear Gelfand triples.
\end{corollary}
\begin{proof}
Formula~\eqref{eq:ThetaSupport} identifies the first induced character
with the MF character \(\Theta_k\) and counts its constituents.
At \(q=0\), the bipartitions in its support have two even-row
components; division of every row by two gives an ordered pair of
partitions of total size \(k\), proving the rank formula.
Finally, \eqref{eq:Bsplit} writes the MF character \(\Theta_1\) as the
sum of the two genuine characters \(U_n\) and \(\ee U_n\).
They are individually MF, and Proposition~\ref{prop:Gelfandcriteria}
applies.
\end{proof}

\begin{remark}[Comparison with the literature]
The constituent formulas and MF property of these triples occur in
the perfect-model construction \cite{MZgraphs,MZperfect}; the
corollary records their finite Gelfand-triple algebras explicitly.
The classification of commutative parabolic Hecke algebras in
\cite{APVM} applies when the stabilizer is standard parabolic.
It does not classify the signed pairing subgroups \(C_k\), so it cannot
replace the QP core reduction below.
\end{remark}

\subsection{Two symmetric-group projections}

For a \(B_n\)-character \(\chi\), define
\[
 \pi_L(\chi)=\sum_{\lambda\vdash n}
       \langle\chi,\chi^{(\lambda,\emptypar)}\rangle\chi^\lambda,
 \qquad \pi_R(\chi)=\pi_L(\ee\chi).
\]
We first ensure that these projections retain the required restriction on inducing characters.

\begin{lemma}[Extension from a QP graph]\label{lem:graphextend}
Let \(L\leq S_n\) and \(\varphi:L\to C_2\).
If its graph is QP in \(S_n\times A_1\), then \(\varphi\) extends to a linear character of a standard Young subgroup containing \(L\).
\end{lemma}
\begin{proof}
If \(\varphi=1\), the trivial character of \(S_n\) extends it.
Otherwise the projection makes \(L\) QP.
Embed \(S_n\times S_2\) as a standard parabolic subgroup of \(S_{n+2}\); parabolic induction preserves the QP condition on the same subgroup.
If \(\varphi\ne1\), the final two points form an orbit without an independent \(S_2\) factor.
In Proposition~\ref{prop:Lu}, only a total-even factor can couple to this two-point orbit.
The diagonal \(S_2\) is precisely the total-even subgroup of \(S_2\times S_2\); all other diagonal, fiber-product, and exceptional factors have larger relevant orbits.
Thus \(\varphi\) is the product of permutation signs on a selection of the consecutive symmetric or pairing factors.
More explicitly, write \(\epsilon_i\) for the permutation parity on each selected natural orbit.
The total-even equation on that factor is \(a+\sum_i\epsilon_i=0\) over \(\F_2\), where \(a\) is the parity on the last two points.
Hence \(\varphi=(-1)^{\sum_i\epsilon_i}\).
Each \(\epsilon_i\) is the restriction of the sign of the full symmetric group on that consecutive orbit, proving the extension assertion.
\end{proof}

\begin{lemma}\label{lem:Bprojection}
Every allowed \(B_n\)-term has at least one nonzero projection \(\pi_L,\pi_R\).
Each nonzero projection is an allowed symmetric-group term.
Consequently, for \(n\geq5\), the two projections of a model are each a complete row or column family.
\end{lemma}
\begin{proof}
Write the term as \(\Ind_H^{B_n}\tau\), with \(\tau=\sigma|_H\), and let \(K=H\cap N_n\).
Frobenius reciprocity gives
\[
 \pi_L(\Ind_H^{B_n}\tau)=
 \begin{cases}
 \Ind_{pH}^{S_n}\bar\tau,&\tau|_K=1,\\
 0,&\tau|_K\ne1,
 \end{cases}
\]
where \(\bar\tau(p(h))=\tau(h)\).
For \(\pi_R\), replace \(\tau\) by \(\tau\ee\).
On the sign kernel of any standard parabolic subgroup, an ambient linear character restricts as either \(1\) or \(\ee\).
Hence at least one projection is nonzero.

An inducing character with no \(\ee\) factor plainly descends to the ambient Young subgroup.
If \(\ee|_K=1\), apply the Coxeter quotient
\[
 B_n\longrightarrow S_n\times A_1,\qquad w\longmapsto(p(w),\ee(w)).
\]
The image of \(H\) is a QP graph.
Lemma~\ref{lem:graphextend} extends its graph character, and multiplication by the remaining permutation signs proves that the descended character is allowed.
Now apply Theorem~\ref{thm:A} to both projections.
\end{proof}

By Lemma~\ref{lem:Acolumns}, a term with projection \(R_q^A\) or \(C_q^A\) has
\begin{equation}
 pH=\pair_k\times S_q,\qquad n=2k+q .
 \label{eq:Bimage}
\end{equation}
In the ranks used in the general proof, the two nonzero projections have the same \(q\).
Indeed the two orbit actions in \eqref{eq:Bimage} determine \(k,q\); the degeneracy \(\pair_1=S_2\) does not give a second pairing/symmetric decomposition in these ranks.

Put
\[
 a_q=\frac{n!}{2^kk!q!},\qquad
 d_q=\Theta_k(1)=2^{q+k}a_q .
\]
For a term with image \eqref{eq:Bimage},
\(\chi(1)=2^na_q/|H\cap N_n|\).
The key estimate is
\begin{equation}
 \chi(1)\leq\frac{d_q}{2}\,
       \#\{\text{nonzero projections of }\chi\}.
 \label{eq:Bbudget}
\end{equation}
After summing over a model, the right side is exactly \(\sum_qd_q=\Gamma_{B_n}(1)\), because each projection contains each layer once.
Thus every estimate in \eqref{eq:Bbudget} must be an equality.

\subsection{Pairing cores and nonempty tails}

Assume \(k\geq2,q\geq2\).
A full-group source has an extreme projection \(E^A_{2k}h_q\) or its transpose, and is not MF.
An unsigned source has parabolic corank at least two.
The remaining maximal parabolic sources are \(B_{2k}\times S_q\) and \(B_q\times S_{2k}\).

In the reverse source, enlarge the subgroup to \(L\times\pair_k\), where \(L\) is its signed-factor projection.
The ambient linear character extends to this overgroup.
But \(\Ind_{\pair_k}^{B_{2k}}\one\) contains \(\chi^{((2k-2),(2))}\) twice, through the \(S_{2k}\)-predecessors \((2k)\) and \((2k-2,2)\).
Its linear twists also have repeated constituents.
Induction of a positive multiple preserves repetition, so the reverse source is excluded.

In the forward source let \(L\) be the projection to \(B_{2k}\).
Goursat's lemma describes a subgroup of a direct product with surjective projections as a fiber product over a common quotient.
Because the entire unsigned image is \(\pair_k\times S_q\), that quotient is a quotient of \(L\cap N_{2k}\) and also of \(S_q\).
It is therefore trivial or \(C_2\).
Besides \(L\times S_q\), one must retain
\begin{equation}
 H=\Ker(\varphi\boxtimes\sgn:L\times S_q\to C_2),
 \qquad \varphi|_{L\cap N_{2k}}\ne1 .
 \label{eq:Bcouple}
\end{equation}

Proposition~\ref{prop:Lu} gives the core list and the following bounds.
The two D forms are kept separately as subgroup structures, even when their kernel orders agree.
\begin{table}[htbp]
\centering\small
\setlength{\tabcolsep}{4pt}
\begin{tabular}{llll}
\toprule
Core or source&Sign-kernel order of \(H\)&Degree bound&Equality in \eqref{eq:Bbudget}\\
\midrule
Unsigned pairing&\(1\)&Not MF&Never\\
Full B form, product&\(2^{2k}\)&\(2^qa_q\), one projection&Never\\
Full B form, \eqref{eq:Bcouple}&\(2^{2k-1}\)&\(2^{q+1}a_q\), two&Never\\
Either D form, product&\(2^{2k-1}\)&\(2^{q+1}a_q\), two&Never\\
Special \(C_k\), product&\(2^k\)&\(d_q\), two&Ordinary layer\\
Exceptional rank-four lift&\(2\)&Projection not a full layer&Never\\
\bottomrule
\end{tabular}
\caption{The pairing-core degree estimate for \(k\geq2,q\geq2\).}
\label{tab:Bcore}
\end{table}

Here is the exclusion of additional couplings in the table.
For a Coxeter-even core, there are no reflections.
A QP graph over an \(A_1\) quotient with nontrivial second projection would require a generating rotation consisting of a core reflection and that \(A_1\) reflection.
This is impossible.
It also excludes the even exceptional rank-four core.
Its uncoupled character either has a repeated same-oriented Pieri projection, or has a second projection which is not a complete row/column layer.

For an even-sign D form with \(k\geq3\), there is no nontrivial \(\pair_k\)-invariant linear character of its sign kernel.
Invariance under within-pair swaps first kills every pair flip.
On the remaining even vector space, invariance under pair permutations kills a vector supported on two pairs: for distinct \(i,j,l\),
\((e_i+e_j)+(e_j+e_l)=e_i+e_l\).
Invariance makes all three values equal, so that value is zero in \(\F_2\).
For \(k=2\), a nontrivial invariant character must be negative on cross-pair flips and positive on pair flips.
If it is positive on every core reflection, its graph is not generated by rotations and simple reflections.
If it is negative on core reflections, its kernel is odd but has no reflection, contradicting Proposition~\ref{prop:QPfacts}(iii).

For \(C_k\), \eqref{eq:Bcouple} would make \(\Ker\varphi\) QP in \(B_{2k}\), with image \(\pair_k\) and kernel order \(2^{k-1}\).
Lu's pairing-image list excludes this even-special B form.
The rank-four exceptional lift is also excluded here because it has elements with \(\ee=-1\), whereas \(\Ker\varphi\leq C_k\leq\Ker\ee\).
This proves all the entries and the strictness assertions in Table~\ref{tab:Bcore}.
The equality case is the ordinary layer, with its two projections in the same orientation.

\subsection{The fixed coordinate in odd rank}

The full-group case \(q=1\) requires a separate estimate.
We give the relevant kernel argument.
Put \(\pair_k^+=\pair_k\cap A_{2k}\) and \(C_k^+=C_k^\circ\).

\begin{lemma}\label{lem:Bsmallkernel}
Suppose \(k\geq3\), \(M\leq B_{2k}\) is even and QP, \(pM\lhd\pair_k\), and \(\pair_k/pM\) is elementary abelian of order at most four.
If \(|M\cap N_{2k}|<2^{2k-1}\), then
\[
 pM=\pair_k^+,\qquad M=\pair_k^+\ \text{or}\ C_k^+
\]
in a permitted unsigned or consecutive signed pairing embedding.
\end{lemma}
\begin{proof}
The commutator subgroup
\([\pair_k,\pair_k]=(C_2^k)^+\rtimes A_k\) is transitive on the \(2k\) points.
Therefore \(pM\) is transitive.
Filtering Lu's symmetric-group list by the normal elementary-abelian quotient leaves only \(\pair_k,\pair_k^+\).
Indeed, the diagonal and fiber-product cases have two orbits.
A symmetric or alternating group on \(2k\) points has too large an order to lie in \(\pair_k\).
The only even-degree transitive exceptional images to check are the degree-six \(A_5\) and degree-eight \(\AGL_3(2)\); their orders have factors five and seven, respectively, absent from \(|\pair_3|\) and \(|\pair_4|\).
A transitive pairing factor on \(2k\) points has order \(|\pair_k|\) or \(|\pair_k|/2\), so containment gives \(\pair_k\) or its permutation-even subgroup.

For these images, the lift construction in Lu's Theorem 16 gives the following exhaustive kernel alternatives before imposing the strict bound:
\[
\begin{array}{c|c}
\text{form}&\text{order of the pure sign kernel}\\ \hline
\text{unsigned}&1\\
\text{special}&2^k\\
\text{either D form}&\text{at least }2^{2k-1}\\
\text{B form}&2^{2k}.
\end{array}
\]
The even-special kernel of order \(2^{k-1}\) is excluded in type B by the same classification.
The exceptional signed lift with pairing image has rank four and is outside \(2k\geq6\).
The unsigned and special lifts have Coxeter parity equal to the permutation parity of their image.
Evenness therefore changes their image to \(\pair_k^+\), giving exactly the two groups in the statement.
\end{proof}

Two concrete QP1 obstructions will be useful.
In \(B_3\), both
\[
 L_f=\langle f_1,(23),f_2f_3\rangle,\qquad
 L_l=\langle(12),f_1f_2,f_3\rangle
\]
fail QP1.
For \(L_f\), the distinct cosets of \(w=(1,-2,3)\) and \((13)w\) both have height \(3\).
For \(L_l\), those of \(w=(-1,3,2)\) and \(f_1f_2(12)w\) both have height \(2\).
Both coset actions have six points.
Successive largest-symbol deletions in Proposition~\ref{prop:QPfacts}(iv) show that
\begin{equation}
 \langle f_1\rangle\times C_k[2,\ldots,2k+1],
 \qquad C_k[1,\ldots,2k]\times\langle f_{2k+1}\rangle
 \label{eq:Blocal}
\end{equation}
are not QP in their indicated embeddings.
Square brackets specify the supporting coordinates.
Likewise
\(\pair_k[1,\ldots,2k]\times\langle f_{2k+1}\rangle\)
reduces to \(\langle f_2\rangle\leq B_2\), an odd subgroup with no simple reflection, and is not QP.

\begin{lemma}[Fixed-coordinate estimate]\label{lem:Bfixed}
Let \(k\geq3\), \(n=2k+1\), and let \(H\leq B_n\) be QP with \(pH=\pair_k\times S_1\).
If a full-group ambient twist of \(\Ind_H^{B_n}\one\) can occur in a model, then either \(H\) is an ordinary initial \(B_{2k}\) core, or
\[
 |H\cap N_n|\geq2^{2k-1}.
\]
\end{lemma}
\begin{proof}
The absolute fixed coordinate \(j\) is \(1\) or \(n\), because the unsigned orbits are consecutive.
Let \(\tau:H\to\{\pm1\}\) record its sign.
If \(\tau=1\), the single-core list applies after deletion.
The only small kernels are unsigned pairing and the initial \(C_k\); the unsigned source is not MF.

Assume \(\tau\ne1\) and put \(H_0=H^\circ\).
If \(\tau|_{H_0}=1\), the odd coset of \(H_0\) contains a simple reflection changing the fixed coordinate.
Thus \(j=1\), this reflection is \(f_1\), and \(pH_0=\pair_k\).
After deletion, Lemma~\ref{lem:Bsmallkernel} forces the large-kernel bound.

Otherwise \(M=H_0\cap\Ker\tau\) is the stabilizer of the largest signed symbol in the \(H_0\)-orbit \(\{j,-j\}\).
Proposition~\ref{prop:QPfacts}(iv), applied to that symbol, identifies \(M\) after deletion as an even QP subgroup of \(B_{2k}\).
It is the intersection of two character kernels in \(H\), so \(pM\lhd\pair_k\) with elementary-abelian quotient of order at most four.
Lemma~\ref{lem:Bsmallkernel} reduces the small-kernel cases to
\(M=\pair_k^+\) or \(C_k^+\).
Because \(H_0\) is generated by rotations, it has a rotation \(f_jr\) changing coordinate \(j\), where \(r\) is a reflection on the remaining coordinates.
Consequently, for \(L=\langle M,r\rangle\),
\[
 [L:M]=2,\qquad
 H_0=\{(l,f_j^{\,\ell(l)\bmod2}):l\in L\}.
\]

If \(M=\pair_k^+\), a signed reflection normalizing \(M\) would have a sign vector fixed by its transitive unsigned image.
A vector of weight one or two cannot be constant on \(2k\geq6\) coordinates.
A cross-pair transposition does not normalize the pairing, as conjugating \(a_i a_l\), with \(l\) a third pair, shows.
Thus \(r\) is an unsigned within-pair swap and \(L=\pair_k\).
If \(M=C_k^+\), normalization of its pair-flip kernel forces \(r\) to preserve the pairing.
A short reflection has a commutator with a pair interchange that is a cross-pair sign flip, outside this kernel.
The remaining positive and negative within-pair reflections lie in the same \(M\)-coset.
Hence \(L=C_k\).

For \(L=C_k\), the group \(L\times\langle f_j\rangle\) has even subgroup \(H_0\) and contains a simple reflection.
It would be QP by Proposition~\ref{prop:QPfacts}, contradicting \eqref{eq:Blocal}.
For \(L=\pair_k\) and \(j=n\), the same argument contradicts the unsigned deletion obstruction.
For \(j=1\), any odd simple extension remains inside \(B_1\times\pair_k\).
Its full-group linear character extends to this overgroup, whose unsigned pairing core has the repeated constituent \(\chi^{((2k-2),(2))}\) already exhibited.
Thus it is not MF.
All small-kernel full-group alternatives are excluded.
\end{proof}

The large-kernel terms in the lemma satisfy
\[
 \chi(1)\leq4a_1<2^ka_1=d_1/2\qquad(k\geq3).
\]
Since at least one projection is nonzero, this proves \eqref{eq:Bbudget}, strictly, for every additional full-group \(q=1\) term.
Proper-parabolic \(q=1\) terms are covered by the same core list as Table~\ref{tab:Bcore}; the only equality case is \(C_k\).

\subsection{Equality and completion}

\begin{proof}[Proof of Theorem~\ref{thm:B}]
For ranks \(4,5,6\), apply Proposition~\ref{prop:finite} to the full group and every maximal standard parabolic subgroup.
The resulting QP MF column counts are \(90,153,219\), respectively, and each column set has exactly four exact covers.
The bipartition supports identify them with the four displayed families.
The exhaustive subgroup inputs, basepoint tests, and character identifications are supplied in \cite{ZhangData}.
Assume \(n\geq7\).
The preceding sections prove the budget for \(k\geq2,q>0\).
For \(q=0\), the pairing-core list gives the same conclusion:
unsigned pairing is not MF; the B and two D forms have respectively one and two projections and strictly smaller degrees; only \(C_k\) attains equality.

For \(k=1\), \(q=n-2\geq5\), the forward source is \(B_2\times S_q\).
The complete rank-two linear-graph list gives either the ordinary double-projection term \(\Theta_1\), the two single-projection terms \(U_n,\ee U_n\), or an equality term with oppositely oriented projections from an index-two B-form graph.
The two nontrivial-kernel graphs of \(C_1\) are not QP; a real linear character of the cyclic order-four core is trivial on its pure sign kernel.
These assertions follow by checking the rank-two coset actions.
In the reverse source, the single \(S_q\)-image list has only the unsigned, B, and two D forms, since \(q\geq5\).
The unsigned case has corank at least two; all others, including their possible binary coupling, have degree at most \(8a_q\), strictly less than their budget.

For \(k=0\), the image is \(S_n\).
The unsigned \(S_n\) term has degree \(2^n=d_n\) and two projections of the same orientation.
Every other lift has degree at most two and strict budget.
Thus equality forces the two entire symmetric-group projections to have the same orientation.
This removes the mixed equality alternative at \(k=1\).

Summing \eqref{eq:Bbudget} now forces every layer with \(k\ne1\) to be the ordinary \(\Theta_k\).
At \(k=1\), one chooses either the whole layer or its split \eqref{eq:Bsplit}.
The common row/column orientation gives the two Coxeter-sign duals.
The four families are distinct by their supports and their numbers of terms.
\end{proof}

\begin{example}[Degrees in ranks four and five]
For \(B_4\), the ordinary layers have degrees
\[
 (\Theta_0(1),\Theta_1(1),\Theta_2(1))=(16,48,12).
\]
The split model has degrees \(16,24,24,12\).
For \(B_5\), the corresponding lists are
\((32,160,120)\) and \((32,80,80,120)\).
These sum to \(76\) and \(312\), respectively.
The split separates two allowed inducing terms without changing the total Gelfand character.
\end{example}

\begin{example}[A repeated constituent detected before induction]
In the unsigned pairing lift to \(B_4\),
\(\chi^{((2),(2))}\) occurs twice in \(\Ind_{\pair_2}^{B_4}\one\).
Its two predecessors in \(S_4\) are \((4)\) and \((2,2)\).
This small example is the same positive-multiplicity obstruction used for every reverse source in the proof.
\end{example}

\subsection{The complete list in type \texorpdfstring{\(B_3\)}{B3}}
\label{subsec:B3classification}

Let \(W=B_3\), with the signed-permutation generators and characters already fixed.
Write \(R\) for the degree-three reflection character, and \(U\) for the inflated degree-two character of \(S_3\).
The ten irreducible characters are
\[
 \one,\dd,\ee,\om,\ U,\ee U,\ R,\ee R,\dd R,\om R ,
 \qquad\dd U=U .
\]
Set \(P=\langle s_1,s_2\rangle\), \(z=-I\), \(K=P\times\langle z\rangle\), and \(W^+=\Ker\om\).
For \(\tau\in\{\one,\dd,\ee,\om\}\), define the allowed characters in Table~\ref{tab:B3terms}.

\begin{table}[htbp]
\centering\small
\begin{tabular}{llr}
\toprule
Term and decomposition&Subgroup for its untwisted term&Degree\\
\midrule
\(A_\tau=\tau(1+U)\)&\(N_3\rtimes\langle(12)\rangle\)&3\\
\(B_\tau=\tau(1+U+R)\)&\(\langle s_0,s_1\rangle\)&6\\
\(E_\tau=\tau(1+U+R+\ee R+\dd R)\)&\(\langle s_0,s_2\rangle\)&12\\
\(F_\tau=\tau(1+\ee R)\)&\(K\)&4\\
\(C=1+\ee+R+\ee R\)&\(P\)&8\\
\(D=1+\ee+U+\ee U+R+\ee R\)&\(C_1=\langle(12),f_1f_2\rangle\)&12\\
\(L_+=1+\om\)&\(W^+\), trivial character&2\\
\(L_-=\dd+\ee\)&\(W^+\), restricted \(\dd\)&2\\
\bottomrule
\end{tabular}
\caption{Characters used in the \(B_3\) models. A subscript \(\tau\) means restriction of that full-group linear twist to the displayed subgroup.}
\label{tab:B3terms}
\end{table}

All displayed subgroups have QP representatives.
For \(K\), the four-point quotient has minimum height sequence \(0,1,2,3\); its standard generators fix the minimum as required in Subsection~\ref{sec:B3canonical}.
The other nonparabolic entries are ordinary signed-pairing or two-point QP sets.
The order-\(16\) subgroup for \(A_\tau\) is the Coxeter-quotient preimage of a reflection subgroup of \(S_3\).
The coefficient formulas in the table can also be checked from the signed coset actions.

\begin{theorem}\label{thm:B3list}
There are twenty-four models in type \(B_3\).
Eight are the following four families and their \(\om\)-twists:
\[
 \begin{gathered}
 \{C,\dd D\},\qquad
 \{C,B_\dd,B_\om\},\\
 \{B_\ee,E_\dd,\one,\om\},\qquad
 \{E_\om,B_1,\dd,\ee\}.
 \end{gathered}
 \]
They have perfect realizations.
The other sixteen are the eight families in Table~\ref{tab:B3extra} and their \(\om\)-twists.
\end{theorem}

\begin{table}[htbp]
\centering
\begin{tabular}{clc}
\toprule
Family&Individual characters&Degrees\\
\midrule
\(\mathscr N_1\)&\(F_1,F_\ee,B_\dd,B_\om\)&\(4,4,6,6\)\\
\(\mathscr N_2\)&\(F_1,F_\ee,\dd D\)&\(4,4,12\)\\
\(\mathscr N_3\)&\(F_1,E_\dd,A_\ee,\om\)&\(4,12,3,1\)\\
\(\mathscr N_4\)&\(F_1,E_\dd,A_\om,\ee\)&\(4,12,3,1\)\\
\(\mathscr N_5\)&\(B_\ee,E_\dd,L_+\)&\(6,12,2\)\\
\(\mathscr N_6\)&\(E_\om,F_\ee,A_1,\dd\)&\(12,4,3,1\)\\
\(\mathscr N_7\)&\(E_\om,F_\ee,A_\dd,\one\)&\(12,4,3,1\)\\
\(\mathscr N_8\)&\(E_\om,B_1,L_-\)&\(12,6,2\)\\
\bottomrule
\end{tabular}
\caption{Representatives of the sixteen additional \(B_3\) models. Each row and its Coxeter-sign twist are distinct.}
\label{tab:B3extra}
\end{table}

\begin{proof}
The character identities in Table~\ref{tab:B3terms} verify that every displayed family contains each of the ten irreducible characters once.
All have degree \(20\), and their individual supports distinguish them.
For completeness, Proposition~\ref{prop:finite} applied to the full group and its three maximal standard parabolic subgroups gives forty distinct QP MF columns.
The exact-cover equations on the ten irreducible rows have precisely twenty-four solutions; their individual supports are the displayed families.
The archive \cite{ZhangData} retains all forty columns, including those occurring in no model, together with the exhaustive subgroup lists and the QP basepoint tests.
Thus this calculation proves completeness of the family list, in addition to verifying the displayed character identities.
The perfect-model classification in \cite{MZperfect} identifies exactly the first eight families as perfect.
Thus the other sixteen are precisely the additional families in Table~\ref{tab:B3extra}.
\end{proof}

\begin{corollary}[The ordinary Gelfand pairs used in type \(B_3\)]
\label{cor:B3Gelfand}
The following seven pairs are Gelfand pairs; the last column is the
dimension of their double-coset algebra:
\[
\begin{array}{c|c|c}
\text{term}&H&\dim\mathscr A(W,H,\one)\\ \hline
A_\one&N_3\rtimes\langle(12)\rangle&2\\
B_\one&\langle s_0,s_1\rangle&3\\
E_\one&\langle s_0,s_2\rangle&5\\
F_\one&K&2\\
C&P&4\\
D&C_1&6\\
L_+&W^+&2
\end{array}
\]
Every term in Theorem~\ref{thm:B3list} is either one of these
permutation characters or a twist by a linear character of \(W\), and
hence is a linear Gelfand triple with the same commutant algebra.
In particular,
\[
 \Ind_{W^+}^{W}(\dd|_{W^+})=\dd+\ee
\]
is a two-dimensional twisted Gelfand term.
\end{corollary}
\begin{proof}
For each untwisted term, Table~\ref{tab:B3terms} gives an MF
permutation character with the indicated number of constituents.
Proposition~\ref{prop:Gelfandcriteria}(i) gives the first statement.
Tensoring an induced character by a linear character of \(W\) twists
its restriction on \(H\) and preserves its endomorphism algebra.
The formula for \(L_-\) follows from \(\om=\ee\dd\).
\end{proof}

\begin{remark}
The perfect-model literature \cite{MZgraphs,MZperfect} already
accounts for the pairs and twists occurring in the eight perfect
families.  Corollary~\ref{cor:B3Gelfand} also makes clear that the
twenty-four model count does not represent twenty-four different
ordinary Gelfand pairs: the additional models are new exact-cover
assemblies of twists of the seven displayed pairs.
The exhaustive statement concerns the QP terms occurring in models;
it is not a classification of every Gelfand pair inside \(B_3\).
\end{remark}

\begin{example}[A natural model with a four-point term]
The family \(\mathscr N_3\) has supports
\[
 \begin{array}{c|l}
 F_1&1+\ee R\\
 E_\dd&\dd+U+\dd R+\om R+R\\
 A_\ee&\ee+\ee U\\
 \om&\om .
 \end{array}
\]
Every irreducible occurs once.
Subsection~\ref{sec:B3canonical} constructs its canonical \(W\)-graph as a disjoint union of four natural graphs.
\end{example}
\section{Type D Coxeter group}
\label{sec:D}

Recall \(E_m=\Ker\eta_m\leq C_m\) from \eqref{eq:CE}.
Write \(\sgn\mathcal M=\{\sgn_{D_n}\chi:\chi\in\mathcal M\}\) and
\(\mathcal M^\diagauto=\{\chi^\diagauto:\chi\in\mathcal M\}\).

\begin{theorem}\label{thm:D}
For odd \(n\geq5\), the two models are
\[
 \mathcal D_n=\{\Phi_k:0\leq k\leq(n-1)/2\},\qquad
 \sgn\mathcal D_n.
\]
For \(n=2m\geq6\) with \(m\) odd, the four models are
\begin{equation}
 \mathcal E_n=\{\Phi_0,\ldots,\Phi_{m-1},\Ind_{E_m}^{D_n}\one\},
 \quad \mathcal E_n^\diagauto,\quad
 \sgn\mathcal E_n,\quad \sgn\mathcal E_n^\diagauto .
 \label{eq:Dmodels}
\end{equation}
For \(4\mid n\), no model exists.
\end{theorem}

\subsection{The even-special character and existence}

\begin{lemma}\label{lem:half}
In the associate convention \eqref{eq:associate-sign},
\[
 \Ind_{C_m}^{D_{2m}}\eta_m=\sum_{\lambda\vdash m}\chi^{[\lambda,+]}.
 \]
Consequently
\begin{equation}
 \Ind_{E_m}^{D_{2m}}\one
 =\Ind_{C_m}^{D_{2m}}\one+\sum_{\lambda\vdash m}\chi^{[\lambda,+]},
 \label{eq:half}
\end{equation}
and this character is MF if and only if \(m\) is odd.
\end{lemma}
\begin{proof}
On one pair, \(\Ind_{C_1}^{B_2}\eta_1\) is the two-dimensional irreducible character.
Wreath induction and the Cauchy identity therefore give
\[
 \Ind_{C_m}^{B_{2m}}\eta_m
       =\sum_{\lambda\vdash m}\chi^{(\lambda,\lambda)} .
\]
By Clifford theory, \(\Ind_{C_m}^{D_{2m}}\eta_m\) chooses exactly one associate from each pair and has no nondegenerate constituent.
It remains to determine the signs.

Evaluate its difference on \(z_m\) and \(z_m^\diagauto\).
The centralizer \(C_{B_{2m}}(z_m)\) is \(C_m\), and it is contained in \(D_{2m}\); thus the \(B_{2m}\)-class splits into these two \(D_{2m}\)-classes, with centralizer order \(|C_m|\) in each.
An element \(h\in C_m\) of this cycle type permutes the \(m\) coordinate pairs as an involution \(w\in S_m\).
On a fixed pair it is \(a_i\) or \(c_i a_i\), giving two choices.
On two interchanged pairs it has the form \((g,g^{-1})\) followed by their interchange, with \(g\in C_2\times C_2\), giving four choices.
If \(w\) has \(j\) transpositions, there are therefore \(2^{m-2j}4^j=2^m\) such elements.

The class sign, positive for the class of \(z_m\), is \(\eta_m(h)\).
To see this directly, replacing \(a_i\) by \(c_i a_i\) on a fixed pair requires a conjugator with odd pure sign parity and changes both signs.
The contribution of an interchanged pair of pairs has even pure sign parity and \(\eta_m\)-value one.
Consequently each \(h\) contributes \(\eta_m(h)^2=1\) to the difference of the induced values; the centralizer-to-subgroup factor is one.
The induced-character formula gives
\[
 \left(\Ind_{C_m}^{D_{2m}}\eta_m\right)(z_m)
 -\left(\Ind_{C_m}^{D_{2m}}\eta_m\right)(z_m^\diagauto)
 =2^m\#\{w\in S_m:w^2=1\}
 =2^m\sum_{\lambda\vdash m}f^\lambda .
\]
The last equality is the symmetric-group Gelfand degree identity, also following from \eqref{eq:Apieri}.
By \eqref{eq:associate-sign}, choosing a minus associate changes its positive contribution to a negative one.
The displayed maximal sum forces all choices to be plus.
This proves the first assertion.
The index-two formula proves \eqref{eq:half}.
By \eqref{eq:Ctop}, overlap occurs exactly for even-row partitions of \(m\).
There are such partitions if and only if \(m\) is even.
\end{proof}

When \(m\) is odd, every partition of \(m\) has a positive number of odd rows.
Formula \eqref{eq:PhiSupport} says that \(\Phi_0,\ldots,\Phi_{m-1}\) contain all minus associates, once each, while \eqref{eq:half} supplies all plus associates.
The nondegenerate characters are partitioned by their total number of odd rows.
This proves existence in \eqref{eq:Dmodels}.
The odd-rank existence follows directly from the same support partition, with no degenerate characters.

\begin{example}[The new top term in \(D_6\)]
Here \(m=3\), \(|C_3|=384\), and \(|E_3|=192\).
The new term has degree \(23040/192=120\).
Its extra half-degenerate part consists of
\(\chi^{[(3),+]},\chi^{[(2,1),+]},\chi^{[(1^3),+]}\), of degrees \(10,40,10\).
The ordinary part also has degree \(60\).
The other three layers have degrees \(32,240,360\), giving a model of degree \(752\).
\end{example}

\subsection{Projection and complete core reduction}

Define
\[
 \pi_D(\chi)=\sum_{\lambda\vdash n}
      \langle\chi,\chi^{\{\lambda,\emptypar\}}\rangle\chi^\lambda .
\]
For \(\Ind_H^{D_n}\tau\), this is zero unless \(\tau\) is trivial on \(H\cap N_n^+\); otherwise it is \(\Ind_{pH}^{S_n}\bar\tau\).
Except in \(D_2\times S_{n-2}\), every ambient linear character directly descends through \(p\).
We record why this exceptional source still gives an allowed A-type projection.

\begin{lemma}\label{lem:Dprojection}
Every nonzero projection \(\pi_D\) of an allowed term is an allowed symmetric-group term.
For \(n\geq5\), these projections in a model form either the complete row family or the complete column family.
A zero-projection MF term must come from \(D_2\times S_{n-2}\) with a linear character taking opposite values on \(t,u\).
\end{lemma}
\begin{proof}
For the exceptional source write \(D_2=\langle t,u\rangle\cong S_2\times S_2\).
A mixed character descends from \(H\) only if \(tu\notin H\).
Embed the two \(S_2\)'s and the tail \(S_q\) in \(S_{q+4}\).
Lu's list permits coupling an artificial two-point orbit only through a total-even factor, including its diagonal \(S_2\) special case.
Both artificial orbits cannot lie in the same such factor, since that would put \(tu\) in \(H\).
Write \(a,b\in\F_2\) for their two parities, \(c=a+b\) for the parity of the merged two-point orbit under \(p\), and \(\epsilon_i\) for the parities of the consecutive tail orbits.
A coupled artificial orbit has an equation \(a=\sum_{i\in I}\epsilon_i\), or the analogous equation for \(b\); distinct couplings use disjoint tail factors.
If \(a\) is free and \(b\) is coupled, then \(a=c+\sum_{i\in I}\epsilon_i\).
The case with \(a,b\) exchanged is identical, and a fixed artificial orbit has parity zero.
Both artificial orbits cannot be independent free factors, since that too would include \(tu\).
Thus in every case both \(a\) and \(b\) are linear combinations of \(c\) and the \(\epsilon_i\).
A mixed character is \((-1)^a\) or \((-1)^b\), multiplied by a tail sign.
It therefore extends to a product of signs on the merged standard Young subgroup.
The Coxeter quotient also makes the subgroup image QP, as required for an allowed A-type term.
Theorem~\ref{thm:A} gives the family assertion.

All other maximal parabolic sources have linear characters trivial on their sign kernel.
The corank obstruction excludes smaller sources.
This proves the last assertion.
\end{proof}

From now on the general-rank argument assumes \(n\geq8\); ranks \(4,5,6,7\) follow by finite computation.
After a global sign twist, take the row projection.
By Lemma~\ref{lem:Acolumns}, its \(q=n-2k<n\) term has
\(pH=\pair_k\times S_q\); the final \(R_n^A\) has image \(S_n\).

For \(k\geq2,q\geq2\), a full-group source has a same-oriented Pieri projection with repetition.
In \(D_{2k}\times S_q\), the signed projection \(L\) is one of
\[
 \pair_k,\qquad N_{2k}^+\rtimes\pair_k,\qquad C_k,\qquad E_k .
\]
Goursat's lemma retains the product and the binary couplings
\begin{equation}
 H=\Ker(\varphi\boxtimes\sgn),\qquad
 \varphi|_{L\cap N_{2k}^+}\ne1 .
 \label{eq:Dcouple}
\end{equation}
The character of this source is a positive sum
\begin{equation}
 \Ind_{L\times S_q}^{D_n}(\one\boxtimes\sgn)
 +\Ind_{L\times S_q}^{D_n}(\varphi\boxtimes\one).
 \label{eq:Dcouplechar}
\end{equation}
For \(k\geq3\), the sign kernels of the D and E forms have no invariant nontrivial linear character.
For the E form this is the even-vector-space argument with three pair indices used in type B.
For the D form, within-pair swaps first kill pair flips, reducing to the same argument.
Thus only \(C_k\) remains in \eqref{eq:Dcouple}.

Its nontrivial sign-kernel character is \(\eta_k\), possibly multiplied by the within-pair and pair-permutation signs.
Induction to \(B_{2k}\) gives one of
\begin{equation}
 \sum_{\lambda\vdash k}\chi^{(\lambda,\lambda)},\qquad
 \sum_{\lambda\vdash k}\chi^{(\lambda,\lambda^\transpose)}.
 \label{eq:DcoupleCauchy}
\end{equation}
The within-pair sign simultaneously transposes the summation variables.
For a horizontal tail \(q\geq3\), the first sum has two predecessors \((k)\), \((k-1,1)\) at
\[
 ((k,1),(k+q-2,1)),
\]
and the second has two at
\(((k+q-2,1),(2,1^{k-1}))\).
Both targets are nondegenerate.
For \(q=2\), the first additional term overlaps the ordinary \(\Phi_k\) at
\(\{(k+1),(k,1)\}\).
The second has a repeated \(\{(k,1),(2,1^{k-1})\}\) for \(k\geq3\).
Hence \eqref{eq:Dcouplechar} is not MF.

For \(k=2\), the complete local linear-graph test is used in place of the invariant-character assertion.
Enumerate the linear characters of the three explicit cores and test the graphs of those nontrivial on their pure sign kernel in the product with \(A_1\), using QP1 and QP2 at every possible basepoint.
The matrices, accepted heights, and rejected-basepoint witnesses are recorded in \cite{ZhangData}.
All such graphs of the D and E forms fail QP; for the C form, only \(\eta_2\) and its Coxeter-sign multiple survive.
They give the first sum in \eqref{eq:DcoupleCauchy}, already excluded.
The quotient \(S_q\to A_1\) by permutation sign makes this one local test valid for all \(q\geq2\).

The direct E-form product also fails MF when \(q>0\).
Its extra associate indexed by \(\lambda=(k)\), followed by vertical Pieri, contains
\[
 \{(k+1,1^{q-1}),(k)\}.
\]
Its total number of odd rows is \(q\), so the same constituent occurs in \(\Phi_k\).
The unsigned product has parabolic corank at least two.
A reverse source \(D_q\times S_{2k}\) contains an induced positive subcharacter from \(L\times\pair_k\).
The unsigned pairing core has the repeated bipartition \(((2k-2),(2))\).
Choose a one-component extension from the \(B_q\)-induction of its other factor.
The resulting repeated target has component sizes \(2k-2+q\) and \(2\), so is nondegenerate and gives the same repetition in \(D_n\).
For \(q=2\), the mixed fork sign is treated by its index-two extensions; the extension supplying the required A projection is the uniform-sign positive term just used.

We have proved that for \(k\geq2,q\geq2\), a nonzero row-projection MF term is an ordinary \(\Phi_k\), or a proper subcharacter arising from the D form.
The latter misses every nondegenerate target for which both row-even remainders are nonempty.

For \(q=1\), if \(H\) fixes the signed fixed coordinate, apply the same core list.
If it changes that sign, its even part contains a generating rotation doing so.
In type \(D\), such a rotation, while fixing the absolute coordinate, is a simultaneous sign change of that coordinate and a core coordinate.
Transitivity of \(\pair_k\) on the core then generates all \(N_n^+\).
Thus \(H=p^{-1}(\pair_k\times S_1)\), whose character is an inflation of \(R_1^A\), and is again a proper subcharacter missing the same mixed targets.
For \(q=0,n=2m\), the only MF top candidates are the D-form inflation, \(C_m\), and \(E_m\).
The unsigned pairing is not MF.
Lemma~\ref{lem:half} determines exactly when the E term is MF.

\subsection{Eliminating the zero-projection terms}

The following support constraints determine the possible zero-projection summands.
Put \(q=n-2\geq6\).
The \(R_q^A\)-projection can lift from \(D_2\times S_q\) to
\[
 \Phi_1,\qquad \Phi_1+Z_{\one},\qquad
 Z_\gamma=\Ind_{D_2\times S_q}^{D_n}(\alpha\boxtimes\gamma),
 \quad\alpha(t)\ne\alpha(u),\quad\gamma\in\{\one,\sgn\}.
\]
Here the second expression is the additional index-two-graph possibility; the notation \(Z_\gamma\) also describes the positive subcharacters of a zero-projection term.
In the reverse source, the only further \(R_q^A\)-lift is
\[
 P=\Ind_{D_q\times S_2}^{D_n}(\sgn\boxtimes\one).
\]
The single-image list with \(q\geq6\) and Goursat's lemma justify this enumeration: its \(D_q\) projection is unsigned \(S_q\) or \(D_q\), and neither has a linear character nontrivial on its pure sign kernel to support a coupling.
The unsigned alternative has corank at least two.
Pieri gives the five nondegenerate constituents of \(P\):
\begin{equation}
 \begin{gathered}
 \{(3,1^{q-1}),\emptypar\},\quad
 \{(2,1^q),\emptypar\},\quad
 \{(2,1^{q-1}),(1)\},\\
 \{(1^q),(2)\},\quad
 \{(1^{q+1}),(1)\}.
 \end{gathered}
 \label{eq:Dreversefive}
\end{equation}
Only the last is in row \(n\).

Each of \(\Phi_1,\Phi_1+Z_{\one},P\) contains
\[
 \chi^{\{(2,1^{n-3}),(1)\}} .
\]
So does \(Z_{\sgn}\).
Hence every term containing \(Z_{\sgn}\) is already excluded.

Consider any zero-projection term.
It contains \(Z_{\one}\) or \(Z_{\sgn}\) by positivity and Frobenius reciprocity.
Its \(S_q\)-projection \(L\) must be transitive, by the corank obstruction.
If \(L\) is even, the two tail signs have the same restriction, so the term contains the excluded \(Z_{\sgn}\).
This treats alternating, even-pairing, and exceptional transitive factors.
If \(L=\pair_{q/2}\), the two summands \(s_{(q)}\) and \(s_{(q-2,2)}\) in its permutation character give two contributions to
\(\{(2),(q-1,1)\}\) after mixed \(D_2\)-induction.
Thus that case is not MF.
The only remaining possibility is \(L=S_q\).
Zero projection requires \(tu\in H\).
Its kernel in \(D_2\) is therefore \(D_2\) or \(\langle tu\rangle\).
The first gives the bare \(Z_{\one}\); the second is a trivial or sign graph.
The trivial graph doubles the same nondegenerate constituents, and the sign graph contains \(Z_{\sgn}\).
Thus the only remaining zero-projection MF term is \(Z_{\one}\), whose nondegenerate component partitions have at most two rows.

Now
\[
 \lambda_0=\{(1^{n-2}),(1,1)\}
\]
lies in row \(n\) and is absent from \(Z_{\one}\), from \(P\), and from every other row candidate.
It forces the \(R_n^A\)-lift to be \(\Phi_0\), rather than the single sign character.
The last constituent of \eqref{eq:Dreversefive} now excludes \(P\).

Take \(q_*=2\) if \(n\) is even and \(q_*=3\) if \(n\) is odd, and set \(k_*=(n-q_*)/2\geq3\).
The target
\[
 \Lambda_*=\{(2k_*-2,1^{q_*}),(2)\}
\]
has two nonempty row-even remainders and a component of at least three rows.
It is absent from the proper D-form subcharacters and from \(Z_{\one}\).
Hence it forces the whole layer \(\Phi_{k_*}\).
For even \(n\), this layer overlaps \(Z_{\one}\) at
\(\{(n-3,1),(2)\}\); for odd \(n\), it overlaps at
\(\{(n-2),(1,1)\}\).
Thus \(Z_{\one}\) and \(\Phi_1+Z_{\one}\) are excluded, forcing \(\Phi_1\).
Every other positive-\(q\) layer is then forced by a target with two nonempty row-even remainders.
This also excludes the full-kernel inflation at \(q=1\).

\subsection{The associate constraint}

\begin{proof}[Completion of Theorem~\ref{thm:D}]
The preceding argument forces all ordinary positive-\(q\) layers.
For odd \(n\), these already cover every irreducible character, giving the row model and its sign dual.

Let \(n=2m\).
For every \(\lambda\vdash m\), an ordinary layer supplies one associate.
The ordinary top \(C_m\) supplies one only for even-row \(\lambda\), and the D-form inflation supplies none.
If \(m\) is even, an E top is not MF by Lemma~\ref{lem:half}.
For instance, \(\lambda=(1^m)\) is not even-row: one associate is supplied by its positive-\(q\) layer, while neither remaining top candidate can supply the other.
Thus an associate necessarily remains uncovered.

If \(m\) is odd, all partitions of \(m\) have positive odd-row count.
The positive-\(q\) terms have supplied one half of all the associate pairs.
Only the E top can supply the complementary complete half.
For each \(\lambda\vdash m\), the unique ordinary layer containing an associate is the layer with \(q=2o_r(\lambda)\).
Choosing the half supplied by the top term therefore forces the complementary embedding of every layer containing associates; independent choices of those embeddings cannot give a cover.
Layers without associates are diagram-invariant.
There are exactly two row families, exchanged by \(\diagauto\), and two column families.
The row and column families are distinct already under \(\pi_D\), by Theorem~\ref{thm:A}; the two row families differ on the top associate choices, as do the two column families.

For \(D_4,D_5,D_6,D_7\), Proposition~\ref{prop:finite} gives respectively \(37,38,74,65\) distinct QP MF columns and \(0,2,4,2\) exact covers.
The character-identification records in \cite{ZhangData} match every cover with the construction just described, including its associate choice.
These finite cases complete the proof.
\end{proof}

\begin{example}[The parity obstruction at the first even ranks]
For \(D_4\), \(m=2\) and \((2)\) is even-row.
Its plus associate occurs in both terms on the right of \eqref{eq:half}, so the E top is not MF.
For \(D_6\), \(m=3\) and no partition of \(3\) is even-row, so the same construction has no overlap.
For \(D_8\), the even-row partitions \((4)\) and \((2,2)\) again create repetitions.
These examples illustrate the associate obstruction proved for every \(m\).
\end{example}

\begin{corollary}[Gelfand pairs and the even-special obstruction]
\label{cor:DGelfand}
For \(n=2k+q\), every ordinary layer
\[
 \Phi_k=
 \Ind_{C_k\times S_q}^{D_n}(\one\boxtimes\sgn_{S_q})
\]
defines a linear Gelfand triple.
In particular, \((D_{2m},C_m)\) is an ordinary Gelfand pair for every
\(m\geq2\).  For the even-special subgroup,
\[
 \boxed{\ (D_{2m},E_m)\text{ is a Gelfand pair}
          \quad\Longleftrightarrow\quad m\text{ is odd}.\ }
\]

More precisely, set
\[
 a_m=\sum_{i=0}^m p(i)p(m-i),\qquad
 d_m=\begin{cases}0,&m\text{ odd},\\p(m/2),&m\text{ even},\end{cases}
 \qquad
 r_m=\frac{a_m+2p(m)-3d_m}{2}.
\]
Then
\begin{equation}
 \mathscr A(D_{2m},E_m,\one)
 \cong \C^{r_m}\times\operatorname{Mat}_2(\C)^{d_m},
 \qquad
 |E_m\backslash D_{2m}/E_m|
 =\frac{a_m+2p(m)+5d_m}{2}.
 \label{eq:Dcommutant}
\end{equation}
Thus the failure in type \(D_{4r}\) consists of exactly \(p(r)\)
matrix blocks of size two.
\end{corollary}
\begin{proof}
The MF assertion for \(\Phi_k\) is formula~\eqref{eq:PhiSupport}; at
\(q=0\) it gives the assertion for \(C_m\).
Lemma~\ref{lem:half} proves the parity criterion for \(E_m\).

It remains to identify the full corner algebra.
In \eqref{eq:Ctop}, write every even-row partition as twice a
partition.  The number of ordered pairs of such partitions with total
size \(2m\) is \(a_m\).  Equal pairs occur only for even \(m\), and
there are then \(d_m=p(m/2)\) of them.  Hence the number of
nondegenerate constituents is \((a_m-d_m)/2\).
Formula~\eqref{eq:half} supplies all \(p(m)\) plus associates; the
\(d_m\) even-row associates occur twice and all others once.
There are therefore \(r_m\) multiplicity-one constituents and
\(d_m\) multiplicity-two constituents.
Proposition~\ref{prop:Gelfandcriteria}(i) gives the first isomorphism
in \eqref{eq:Dcommutant}; taking dimensions gives the double-coset
formula.
\end{proof}

\begin{example}
The algebras for the first three even ranks are
\[
 \mathscr A(D_4,E_2,\one)\cong\C^3\times\operatorname{Mat}_2(\C),
 \qquad
 \mathscr A(D_6,E_3,\one)\cong\C^8,
\]
\[
 \mathscr A(D_8,E_4,\one)
 \cong\C^{12}\times\operatorname{Mat}_2(\C)^2.
\]
Thus the 120-point action of \(D_6\) on \(D_6/E_3\) has a
commutative orbital algebra of rank eight.
\end{example}

\begin{remark}[Comparison with the literature]
The \(C_m\)-frame action and the characters \(\Phi_k\) occur in the
perfect-model formulas of \cite{MZgraphs,MZperfect}; the QP geometry of
maximal orthogonal-root sets is also developed in \cite{GXroots}.
Those sources treat the frame stabilizer \(C_m\), whereas the
index-two even-special subgroup \(E_m\) is the extra source in this
paper.  Neither the parity criterion nor the matrix-block formula
\eqref{eq:Dcommutant} is stated in those cited forms; here both are
strict consequences of Lemma~\ref{lem:half} and \eqref{eq:Ctop}.
This is not asserted to classify all Gelfand pairs of \(D_{2m}\).
\end{remark}
\section{Exceptional type Coxeter group}
\label{sec:exceptional}

\subsection{Dihedral groups}

Let
\(W=I_2(m)=\langle s,t:s^2=t^2=(st)^m=1\rangle\), \(m\geq3\), and put \(r=st\).
Its two-dimensional irreducible characters are \(\rho_h\), \(1\leq h<m/2\), characterized by
\(\rho_h(r^j)=2\cos(2\pi hj/m)\) and \(\rho_h(r^js)=0\).
For odd \(m\), the linear characters are \(\one,\sgn\).
For even \(m\), write \(\lambda_{ab}\), \(a,b\in\{\pm1\}\), for the linear character with values \(a,b\) on \(s,t\).
A pair of linear characters \(\{\lambda,\mu\}\) may be \emph{split}, as two terms, or \emph{joined}, as the one allowed term
\[
 \lambda+\mu=\Ind_{\Ker(\lambda\mu)}^W(\lambda|_{\Ker(\lambda\mu)}).
\]
The two-point quotient by the nontrivial character \(\lambda\mu\) is QP.

\begin{theorem}\label{thm:dihedral}
The complete dihedral models are the following.
\begin{enumerate}[label=(\roman*)]
\item If \(m\) is odd, choose \(\Ind_{\langle s\rangle}^W\one\) or
\(\Ind_{\langle s\rangle}^W\sgn|_{\langle s\rangle}\), and append the missing linear character. There are two models.
\item If \(m\) is even, choose one of the four distinct reflection-induced characters
\(\Ind_{\langle u\rangle}^W\xi\), \(u=s,t\), \(\xi(u)=\pm1\).
Append the two missing linear characters, either split or joined. There are eight models.
\item If \(m\equiv2\pmod4\), there are eight additional models.
Put \(z=r^{m/2}\).
Choose a character induced from \(\langle z,s\rangle\) that contains all odd-indexed \(\rho_h\) and one of \(\lambda_{+-},\lambda_{-+}\).
Choose a second such character containing all even-indexed \(\rho_h\) and one of \(\lambda_{++},\lambda_{--}\).
Append the remaining two linear characters, split or joined.
\end{enumerate}
Thus the counts are \(2,8,16\), according as \(m\) is odd, divisible by four, or congruent to two modulo four.
\end{theorem}
\begin{proof}
Every subgroup is cyclic in \(\langle r\rangle\), or has the form
\(H=\langle r^d,r^as\rangle\), \(d\mid m\).
Put \(L=m/d\).
For a linear character \(\sigma\) of \(W\), direct averaging gives
\[
 \langle\Ind_H^W(\sigma|_H),\rho_h\rangle=
 \begin{cases}
 1,&\sigma(r)^d=1,\ h\equiv0\pmod L,\\
 1,&\sigma(r)^d=-1,\ h\equiv L/2\pmod L,\\
 0,&\text{otherwise}.
 \end{cases}
\]
The second line requires even \(L\).
For a cyclic rotation subgroup the nonzero two-dimensional multiplicity is twice this value, so an MF cyclic source cannot cover a two-dimensional character.
The proper-parabolic sources are just reflection subgroups and are already included.

Covering \(\rho_1\) forces \(L=1\), or \(L=2\) with \(\sigma(r)^d=-1\).
For odd \(m\) only the reflection case is possible.
For \(4\mid m\), the \(L=2\) case has even \(d\) and is impossible.
This gives (i)--(ii).
For \(m\equiv2\pmod4\), the additional \(L=2\) case is possible and covers exactly the odd indices.
Covering \(\rho_2\) then requires the even \(L=2\) part: \(L=1\) overlaps, and \(L=4\) cannot divide \(m\).
Each parity part contains precisely one linear character of the indicated parity.
There are two choices for each part and two ways to finish, giving eight additional families.
No further term can be added because every irreducible has already been covered.

The reflection subgroups are standard parabolic.
The order-four subgroups in (iii) are preimages of reflection subgroups under the Coxeter quotient
\(I_2(m)\to I_2(m/2)\).
They are QP by Proposition~\ref{prop:QPfacts}.
Their inducing characters are restrictions of the indicated \(\lambda_{ab}\).
All linear pairs are two-point QP quotients, proving existence.
\end{proof}

\begin{corollary}[Dihedral Gelfand pairs]
\label{cor:dihedralGelfand}
For every reflection \(u\in I_2(m)\), the pair
\((I_2(m),\langle u\rangle)\) is a Gelfand pair, and
\[
 \dim\mathscr A(I_2(m),\langle u\rangle,\one)
 =\begin{cases}(m+1)/2,&m\text{ odd},\\m/2+1,&m\text{ even}.
 \end{cases}
\]
If \(m\equiv2\pmod4\), put \(z=(st)^{m/2}\).  Then
\((I_2(m),\langle z,u\rangle)\) is a Gelfand pair of rank
\((m+2)/4\).
Finally, every index-two subgroup \(K\) of \(I_2(m)\) gives a
rank-two Gelfand pair \((I_2(m),K)\).
\end{corollary}
\begin{proof}
The inducing characters in Theorem~\ref{thm:dihedral} are restrictions
of linear characters of \(W\).  Twisting back therefore turns each MF
term into the corresponding MF permutation character.
A reflection-induced character has one linear constituent and all
\((m-1)/2\) two-dimensional constituents when \(m\) is odd, and two
linear constituents and all \(m/2-1\) two-dimensional constituents
when \(m\) is even.
For \(m\equiv2\pmod4\), an order-four source has one linear
constituent and \((m-2)/4\) two-dimensional constituents.
The index-two assertion follows from
\(\Ind_K^W\one=\one+\nu\), where \(\Ker\nu=K\).
Proposition~\ref{prop:Gelfandcriteria}(i) now gives all the ranks.
\end{proof}

\begin{remark}
These dihedral pairs are elementary instances of the finite
double-coset criterion and are also implicit in the dihedral
perfect-model analysis of \cite{MZperfect}.
The contribution of Theorem~\ref{thm:dihedral} is the complete
exact-cover assembly, including the extra order-four sources when
\(m\equiv2\pmod4\), rather than the discovery of the reflection
Gelfand pairs themselves.
\end{remark}

\begin{example}[An additional \(I_2(6)\) model]
With \(z=(st)^3\), one additional family has supports
\[
 \{\rho_1+\lambda_{+-},\ \rho_2+\lambda_{++},\
          \lambda_{-+}+\lambda_{--}\}.
\]
The first two terms are induced from \(\langle z,s\rangle\), using \(\lambda_{+-}\) and \(\lambda_{++}\).
The last term may instead be split into its two linear characters.
The degrees are \(3,3,2\), or \(3,3,1,1\).
\end{example}

\subsection{The icosahedral group}

Let \(W=H_3\), with simple generators \(s,t,u\) and
\(m_{st}=5,m_{tu}=3,m_{su}=2\).
Put \(\om=\sgn_W\) and define
\[
 A=\Ind_{\langle s,t\rangle}^W\one,\qquad
 B=\Ind_{\langle t,u\rangle}^W\one,\qquad
 C=\Ind_{\langle s,u\rangle}^W\alpha,\quad\alpha(s)=1,\ \alpha(u)=-1 .
\]
The two choices of the mixed character give the same \(C\).

\begin{theorem}\label{thm:H3}
The four models of \(H_3\) are
\[
 \{A,\om B\},\qquad \{\om A,B\},\qquad
 \{C,\one,\om\},\qquad \{C,\one+\om\}.
\]
The joined linear term is \(\Ind_{\Ker\om}^W\one\).
\end{theorem}
\begin{proof}
Use \(H_3\cong C_2\times A_5\), and label the five \(A_5\)-characters by
\(1,3,3',4,5\).
A superscript \(+\) or \(-\) records the action of the central involution.
Then
\[
 A=1^++3^-+(3')^-+5^+,\qquad
 B=A+4^++4^-,\qquad C=\Gamma_W-1^+-1^- .
\]
These identities follow by averaging on the three stated parabolic subgroups.
They prove the four asserted exact covers and give degrees
\(12+20=32\) and \(30+1+1=32\).
The two-point quotient proves the QP realization of the joined term.

For completeness, apply Proposition~\ref{prop:finite} to \(W\) and its three maximal standard parabolic subgroups.
There are thirteen distinct QP MF columns, with exactly four exact covers on the ten irreducible rows.
Their supports are the four families above.
The complete subgroup enumeration, character columns, and QP basepoint tests are recorded in \cite{ZhangData}.
\end{proof}

\begin{corollary}[Gelfand pairs and a genuinely twisted term in \(H_3\)]
\label{cor:H3Gelfand}
The parabolic pairs
\[
 (H_3,\langle s,t\rangle),\qquad
 (H_3,\langle t,u\rangle)
\]
are Gelfand pairs of ranks four and six, respectively, and
\((H_3,\Ker\om)\) is a rank-two Gelfand pair.
The triple
\[
 (H_3,\langle s,u\rangle,\alpha),
 \qquad \alpha(s)=1,\quad\alpha(u)=-1,
\]
is a linear Gelfand triple of rank eight, whereas the ordinary pair
\((H_3,\langle s,u\rangle)\) is not a Gelfand pair.
\end{corollary}
\begin{proof}
The decompositions of \(A,B,C\) in the preceding proof give the first
three ranks and the twisted rank.
For the last assertion, identify \(H_3\cong C_2\times A_5\) as above.
The subgroup \(\langle s,u\rangle\) projects to a Klein four group in
\(A_5\), and direct averaging on its three involutions gives
\[
 \left\langle\Ind_{\langle s,u\rangle}^{H_3}\one,5^+\right\rangle=2.
\]
Thus its untwisted corner algebra has a noncommutative
\(\operatorname{Mat}_2(\C)\) block.
\end{proof}

\begin{remark}
The MF identities are already implicit in the perfect-model
classification \cite{MZperfect}.  The corollary isolates their
Gelfand-pair content and records why the mixed character in \(C\)
cannot be replaced by the trivial character; this is a genuine twisted
Gelfand phenomenon.
\end{remark}

\subsection{The remaining exceptional types}

\begin{theorem}[Exceptional obstruction]\label{thm:exceptional}
For \(W\) of type \(H_4,F_4,E_6,E_7,E_8\), there is no family
\[
 \sum_i\Ind_{H_i}^W(\sigma_i|_{H_i})=\Gamma_W,
 \qquad H_i\leq W_{J_i},\quad \sigma_i\in\Lin(W_{J_i}),
\]
even without a QP requirement on \(H_i\).
\end{theorem}
\begin{proof}
We use Proposition~\ref{prop:finite}, with the QP test omitted.
This is essential for the stronger assertion in the statement: the input consists of arbitrary subgroups of the full group and every maximal standard parabolic subgroup.
The exact group and character calculations described below are recorded, together with ordinary GAP reconstruction sources, in \cite{ZhangData}.

For \(H_4,F_4,E_6\), the complete full-group enumerations have respectively \(181,246,350\) subgroup conjugacy classes.
All maximal parabolic subgroups are enumerated in the same way.
For \(E_7\), write
\[
 W=\langle z\rangle\times K,\qquad z=-I,\quad K\cong\Sp_6(2).
\]
If \(L\leq K\) is the projection of \(H\leq W\), then either
\[
 H=\langle z\rangle\times L,
 \quad\text{or}\quad
 H=\{(\varphi(l),l):l\in L\},\qquad\varphi\in\Hom(L,C_2).
\]
Indeed, the kernel of that projection is either \(\langle z\rangle\) or trivial.
In the graph case, the two central eigenspaces of the permutation representation have characters
\[
 \one\boxtimes\Ind_L^K\one,
 \qquad\sgn_{C_2}\boxtimes\Ind_L^K\varphi.
\]
Consequently \(\Ind_L^K\one\) must be MF.
The complete table of marks of \(\Sp_6(2)\) has \(1369\) subgroup classes, of which \(31\) have this property.
All binary homomorphisms from these subgroups give \(86\) full-group representatives by the displayed construction, \(84\) of which pass the full-group MF test.
This is an exhaustive relevant list, although it is not the list of all subgroup classes of \(E_7\).
The seven maximal standard parabolic subgroups are treated by complete subgroup enumeration.

After all ambient linear twists and parabolic sources are included and equal characters are identified, the results are
\[
\begin{array}{c|r|r|l}
 W&\text{all MF columns}&\text{QP MF columns}&
   \text{row absent from all MF columns}\\ \hline
 H_4&43&11&\text{the degree-48 character}\\
 F_4&168&67&\text{the degree-16 character}\\
 E_6&66&17&\text{none}\\
 E_7&140&27&\text{both degree-512 characters}.
\end{array}
\]
The degree-48 and degree-16 characters in their respective groups are unique.
The absent rows prove nonexistence for \(H_4,F_4,E_7\), even in the enlarged class.
For comparison, the QP columns of \(E_6\) omit its unique degree-10 character.

For the stronger \(E_6\) assertion, use the fixed ordering
\(\rho_1,\ldots,\rho_{25}\) of the full character table in \cite{ZhangData}, with degrees
\[
\begin{gathered}
1,1,6,6,10,15,15,15,15,20,20,20,24,24,30,30,\\
60,60,60,64,64,80,81,81,90.
\end{gathered}
\]
Among its \(66\) columns, the only two containing row \(20\) have supports
\[
\begin{aligned}
 A&=\{3,6,10,15,20,24\},\\
 B&=\{3,4,6,8,10,11,15,16,20,21,23,24\}.
\end{aligned}
\]
Every column containing row \(5\) intersects \(B\), so a cover cannot use \(B\).
After choosing \(A\), the only disjoint column containing row \(12\) is
\(C=\{2,8,12,13,19\}\).
Every column containing row \(7\) intersects \(A\cup C\).
Both choices for row \(20\) are therefore impossible.
These four support tests give an explicit exact-cover obstruction; the full character values in the archive, rather than the repeated degrees in the display, specify the row labels.

It remains to prove the \(E_8\) assertion.
Put \(z=-I\) and let \(\chi\) be an allowed term with the QP condition still omitted.
A proper parabolic subgroup fixes a nonzero vector, so it does not contain \(z\); such a term has \(\chi(z)=0\).
For a full-group source, both linear characters of \(W(E_8)\) have value one at \(z\), and centrality gives
\begin{equation}
 \chi(z)=
 \begin{cases}
 [W:H],&z\in H,\\
 0,&z\notin H.
 \end{cases}
 \label{eq:E8centralterm}
\end{equation}
A full-group term induced using \(\sigma\) contains \(\sigma\) once.
Since \(W(E_8)\) has just two linear characters, an exact cover contains at most two full-group terms.
The exact character table gives
\[
 \Gamma_W(1)=199952,\qquad\Gamma_W(z)=15120.
\]
Hence \(15120\) must be one index, or a sum of two indices, of subgroups containing \(z\) whose permutation characters are MF.

Write
\[
 G=W/\langle z\rangle\cong O_8^+(2):2,
 \qquad L=G'\cong O_8^+(2),\qquad |L|=174182400.
\]
The table of marks of \(L\) has \(11171\) subgroup classes.
For \(U\leq G\), either \(U\leq L\) and \([G:U]=2[L:U]\), or
\([G:U]=[L:U\cap L]\).
The subgroup orders in that complete table therefore give the following containing set of possible indices \([G:U]\leq15120\):
\begin{equation}
\begin{gathered}
1,2,120,135,240,270,960,1080,1120,1575,1920,2025,\\
2160,2240,3150,3360,3780,4050,4320,4480,4725,6300,\\
6720,7560,8640,9450,11200,12096,12600,13440,14175,14400,15120.
\end{gathered}
\label{eq:E8indices}
\end{equation}
The only possibilities for the required central sum are \(15120\) and \(7560+7560\).
It suffices to exclude MF permutation characters at these two indices.

If \(U\leq L\), its index in \(L\) is \(3780\) or \(7560\).
The table has respectively \(3\) and \(12\) such subgroup classes, and direct scalar products with its irreducible character table show that every corresponding permutation character has a repeated constituent.
Induction to \(G\) preserves that repetition.

Otherwise put \(K_0=U\cap L\).
Then \([L:K_0]\) is \(7560\) or \(15120\), and \([U:K_0]=2\).
In particular \(U\leq N_G(K_0)\), so every such \(U\) is the preimage of an order-two subgroup of \(N_G(K_0)/K_0\) whose preimage is not contained in \(L\).
The table gives \(12\) and \(6\) possibilities for \(K_0\), respectively.
Enumerating all those normalizer quotients and their order-two subgroups gives the complete relevant extension list:
\[
\begin{array}{c|r|r|r}
 [G:U]&\text{extension representatives}&
  \text{maximum multiplicity }2&\text{maximum multiplicity }3\\ \hline
 7560&7&6&1\\
 15120&6&4&2.
\end{array}
\]
Thus no extension has an MF permutation character.
The finite computations here use the full tables of marks and exact scalar products, and the archive records all subgroup orders, the normalizers, their extensions, and the multiplicity vectors \cite{GAP,CTblLib,TomLib,ZhangData}.
Both alternatives in \eqref{eq:E8centralterm} are impossible, completing the proof.
\end{proof}

\begin{corollary}[Exceptional QP Gelfand pairs]
\label{cor:exceptionalGelfand}
For \(W\) of type \(H_4,F_4,E_6,E_7\), the full-support QP subgroups
\(H\) for which \((W,H)\) is a Gelfand pair have the following indices,
up to subgroup conjugacy admitting a QP minimum point:
\[
\begin{array}{c|l}
 H_4&1,2,10,72,144\\
 F_4&1,2^{(3)},3^{(2)},4,6^{(5)},8,9,16,18,32^{(2)}\\
 E_6&1,2,45,90,135\\
 E_7&1,2,135,270,315,630^{(3)},1260.
\end{array}
\]
Here \(d^{(a)}\) denotes \(a\) conjugacy classes of index \(d\),
rather than a power.

There are also explicit root-frame Gelfand pairs in the two largest
exceptional Weyl groups.  Let \(X_7\) and \(X_8\) be the Green--Xu
sets of maximal orthogonal positive-root sets, and let
\(M_i\) be a point stabilizer in \(W(E_i)\).  Then
\[
 |X_7|=135,\qquad |X_8|=2025,
\]
\[
 |M_7|=21504,\qquad |M_7^\circ|=10752,
 \qquad |M_8|=344064.
\]
These are full-support QP subgroups, and the pairs
\((W(E_7),M_7)\), \((W(E_7),M_7^\circ)\), and
\((W(E_8),M_8)\) are Gelfand pairs.  With all displayed constituents
distinct, their permutation characters are as follows.  Let \(A_7\)
join two frames sharing three roots, and let \(A_8\) join two frames
sharing four roots.  Write \(\rho_{d,\theta}\) for the character of the
\(\theta\)-eigenspace of \(A_i\), where \(d\) is its dimension.  Then
\begin{align}
 \pi_7:=\Ind_{M_7}^{W(E_7)}\one
   &=\rho_{1,14}+\rho_{15,-7}+\rho_{35,5}+\rho_{84,-1},
   \label{eq:E7framecharacter}\\
 \Ind_{M_7^\circ}^{W(E_7)}\one
   &=\pi_7+\sgn_{W(E_7)}\pi_7,
   \label{eq:E7evenframecharacter}\\
 \Ind_{M_8}^{W(E_8)}\one
   &=\rho_{1,28}+\rho_{50,-14}+\rho_{50,10}+
     \rho_{84,13}+\rho_{168,-2}+\rho_{700,-5}+\rho_{972,3}.
   \label{eq:E8framecharacter}
\end{align}
In particular, the corresponding orbital algebras have dimensions
\(4,8,7\), respectively.
\end{corollary}
\begin{proof}
The first table is exactly the sublist of the finite enumeration in
the proof of Theorem~\ref{thm:exceptional} for which the QP test
succeeds and the permutation character is MF.  The archive records
the generators, accepted basepoints, and exact scalar products for
every class.

The QP structures on \(X_7,X_8\) and their transitivity are proved in
\cite{GXroots}.  Orbit--stabilizer gives the three displayed subgroup
orders; for \(M_7^\circ\), use \([M_7:M_7^\circ]=2\), proved below.
If \(R\) is the frame and \(A_R\) is generated by its commuting root
reflections, then \(M_i=N_{W(E_i)}(A_R)\).  For each \(i\), the subgroup
\(A_R\leq M_i\) has zero fixed space; the even subgroup
\(A_R\cap M_7^\circ\) also has zero fixed space.  Hence all three
stabilizers have full support.  The minimum point in the transitive QP action gives the
QP assertion, and Proposition~\ref{prop:QPfacts}(iii) gives it for
\(M_7^\circ\).
The exact orbital calculations in \cite{ZhangData} give four orbitals
for \(X_7\) and seven for \(X_8\).  A generating orbital adjacency
matrix has spectra
\[
 14^{(1)},\ 5^{(35)},\ (-1)^{(84)},\ (-7)^{(15)}
\]
and
\[
 28^{(1)},\ 10^{(50)},\ (-14)^{(50)},\ 13^{(84)},\
 (-2)^{(168)},\ (-5)^{(700)},\ 3^{(972)},
\]
respectively.  Thus the adjacency matrix supplies, respectively, four
and seven nonzero invariant eigenspaces.  On the other hand, the
dimension of the permutation commutant is the number of orbitals,
namely four and seven.  If the multiplicities of the irreducible
constituents are \(m_\rho\), Proposition~\ref{prop:Gelfandcriteria}(i)
therefore gives
\(\sum_\rho m_\rho^2=4\), respectively \(7\).  The displayed number of
nonzero invariant eigenspaces forces equality term by term: every
eigenspace is one irreducible constituent and all multiplicities are
one.  This proves \eqref{eq:E7framecharacter} and
\eqref{eq:E8framecharacter}; the two degree-\(50\) eigenspaces in the
second spectrum are distinct constituents.

For type \(E_7\), the central negative element lies in \(M_7\) and has
odd Coxeter sign.  Induction from the index-two subgroup
\(M_7^\circ\) therefore gives \(\pi_7+\sgn_{W(E_7)}\pi_7\).
The two supports have opposite central eigenvalues and are disjoint,
which proves \eqref{eq:E7evenframecharacter}.
\end{proof}

\begin{remark}[Comparison with the literature]
Green and Xu establish the QP root-frame actions and their level
statistics in \cite{GXroots,GXFano}; their statements do not include
the permutation-character decompositions
\eqref{eq:E7framecharacter}--\eqref{eq:E8framecharacter}.
After quotienting by the central negative element, the \(135\)-point
\(E_7\) action is the classical rank-four dual polar action of
\(\Sp_6(2)\), so its Gelfand-pair and spectral content also occurs in
the literature on dual polar association schemes \cite{BCN}.
Winter and van Luijk had already obtained the transitive \(2025\)-point
\(E_8\) action and the stabilizer order \(344064\) in their study of
the \(E_8\) root graph \cite{WinterVanLuijk}.  The commutativity and
the spectrally labelled decomposition in
\eqref{eq:E8framecharacter} are the additional conclusions supplied
by the exact orbital computation here.

The classification of \cite{APVM} concerns standard parabolic corner
algebras and does not imply the full-support nonparabolic table above.
We are not aware of an earlier published classification in precisely
the class of full-support QP Gelfand pairs.
The table makes no assertion about QP subgroups with non-MF
permutation characters or about all Gelfand pairs of the exceptional
Coxeter groups.  Likewise, the existence of the frame pairs is fully
compatible with Theorem~\ref{thm:exceptional}: a commutative orbital
algebra does not provide the disjoint exact cover required for a
Gelfand model.
\end{remark}
\section{Hecke realization and canonical bases}
\label{sec:hecke}

We now construct the representations corresponding to the classified characters.
The distinction between field-level realizations and natural induction lattices will be maintained throughout.
All canonical bases use \eqref{eq:canonical-def}.

\subsection{General constructions and classical models}

\begin{lemma}\label{lem:canonicalalgorithm}
Suppose a free \(\A\)-module has a triangular bar structure on a finite ordered standard basis.
Then its canonical basis exists and is obtained by the following finite recursion.
Having constructed \(C_y\) for \(y<x\), write
\[
 \overline{m_x}-m_x=\sum_{y<x}b_{y,x}C_y .
\]
For a Laurent polynomial \(b\), let \(b_{<0}\) denote its strictly negative-degree part.
Then
\begin{equation}
 C_x=m_x+\sum_{y<x}(b_{y,x})_{<0}C_y .
 \label{eq:canonicalalgorithm}
\end{equation}
\end{lemma}
\begin{proof}
Applying bar to the left side negates it, so
\(b_{y,x}+\overline{b_{y,x}}=0\).
It follows that
\((b_{y,x})_{<0}-\overline{(b_{y,x})_{<0}}=b_{y,x}\),
which proves bar invariance of \eqref{eq:canonicalalgorithm}.
Its lower coefficients contain only negative powers.
For uniqueness, take the highest nonzero standard coefficient in the difference of two proposed canonical vectors.
It is both bar invariant and a strictly negative-degree polynomial, a contradiction.
\end{proof}

For an ordinary QP module in Proposition~\ref{prop:barinput}, take a strictly increasing simple-reflection path
\(x=s_d\cdots s_1x_0\).
Its bar is explicitly
\begin{equation}
 \overline{m_x}=T_{s_d}^{-1}\cdots T_{s_1}^{-1}m_{x_0}.
 \label{eq:pathbar}
\end{equation}
The bar theorem guarantees path independence.
Combining \eqref{eq:ordinaryM}, \eqref{eq:pathbar}, and Lemma~\ref{lem:canonicalalgorithm} computes all matrices over integer Laurent polynomials.

Let
\[
 \mu(y,x)=[v^{-1}]\,[m_y]C_x.
\]
For \(M(X)\), put \(\tau(x)=\{s:h(sx)>h(x)\}\).
For \(N(X)\), put \(\tau(x)=\{s:h(sx)\geq h(x)\}\).
Marberg's canonical multiplication theorem gives \eqref{eq:Wgraph-def}, with
\begin{equation}
 \omega(x,y)=
 \begin{cases}
 \mu(y,x)+\mu(x,y),&\tau(x)\not\subseteq\tau(y),\\
 0,&\tau(x)\subseteq\tau(y).
 \end{cases}
 \label{eq:canonicalweights}
\end{equation}
For parabolically induced perfect modules, the label consists of strict ascents together with weak fixed points having negative scalar; the same construction is supplied by \cite{MZgraphs,HY1,HY2}.
These formulas specify the convention of every graph in this paper.

\begin{example}[A two-point QP module]
Suppose each simple reflection exchanges \(x_0,x_1\), with heights \(0,1\).
On the ordered standard basis \(m_0,m_1\), every generator has matrix
\[
 A=\begin{pmatrix}0&1\\1&a\end{pmatrix}.
\]
We have
\[
 \bar m_0=m_0,\qquad\bar m_1=m_1-am_0,\qquad
 C_0=m_0,\quad C_1=m_1+v^{-1}m_0.
\]
The labels are \(\tau(0)=S,\tau(1)=\emptypar\), and the graph has one edge \(0\to1\) of weight one.
Indeed \(T_sC_0=-v^{-1}C_0+C_1\) and \(T_sC_1=vC_1\).
\end{example}

\begin{theorem}\label{thm:naturalclassical}
Every family in Theorems~\ref{thm:A}, \ref{thm:B}, and \ref{thm:D} has a realization on natural induction lattices with a compatible bar, canonical basis, and integral canonical \(W\)-graph.
\end{theorem}
\begin{proof}
For type A, for type B in rank at least four, and for odd type D, choose the perfect realizations of the classified character families.
Their Hecke, bar, and canonical \(W\)-graph structures are the constructions of \cite{MZgraphs}.
This includes the small additional \(S_4\) families and the split B family.

For \(n=2m\) with \(m\) odd, let \(J_k\) be the standard parabolic of type \(D_{2k}\times S_{n-2k}\), interpreting \(J_0\) as the unsigned \(S_n\).
Take
\begin{equation}
\begin{split}
 \mathcal G_{D_n}
 ={}&\bigoplus_{k=0}^{m-1}
 \HH(D_n)\otimes_{\HH(W_{J_k})}
 \bigl(M(D_{2k}/C_k)\boxtimes\sgn_{\HH(S_{n-2k})}\bigr)\\
 &\hspace{25mm}\oplus M(D_n/E_m).
\end{split}
\label{eq:DHecke}
\end{equation}
The \(k=0\) tensor factor is simply the Hecke sign module of \(S_n\).
The preceding summands are perfect parabolic constructions.
The last is an ordinary QP module, since \(E_m\) is QP by Lu's even-special classification.
Proposition~\ref{prop:barinput}, parabolic induction, and
\eqref{eq:pathbar}--\eqref{eq:canonicalweights} give all the required structures.
At \(v=1\), \eqref{eq:DHecke} is the direct sum of the specified natural induction lattices.
Its character is \(\Gamma_{D_n}\) by Theorem~\ref{thm:D}.
The Tits deformation correspondence gives the generic Hecke Gelfand property.

Apply the diagram automorphism to obtain the diagram-dual family.
For sign-dual families use the corresponding \(N\) and dual perfect constructions, with their own canonical bases.
They again specialize to the sign-twisted natural lattices.
\end{proof}

The theorem asserts the existence of a natural realization of each individual-character family.
Canonical-basis uniqueness applies after the chosen standard basis, order, and bar are fixed.
It is not a uniqueness assertion about all triples representing the same character.

\begin{example}[The \(D_6\) realization]
For \(n=6\), the four summands of \eqref{eq:DHecke} have ranks \(32,240,360,120\), respectively.
The final summand is the ordinary QP module \(M(D_6/E_3)\).
Their direct sum is a \(752\)-dimensional Gelfand module, with bar and canonical basis supplied by the preceding construction.
\end{example}

\begin{corollary}[Generic realization from Gyoja's theorem]\label{prop:allHecke}
Every individual-character model in Theorem~\ref{thm:main} has a termwise generic Hecke realization over a splitting field, with a compatible bar and a \(W\)-graph.
\end{corollary}
\begin{proof}
Choose, by Proposition~\ref{prop:Gyoja}, a \(W\)-graph for every irreducible constituent.
For each inducing term, take the disjoint union of the graphs in its support.
The exact-cover property makes the direct sum a Gelfand module.
Coefficient conjugation fixing the graph vertices gives the bar.
With the graph basis itself as standard basis and the discrete order, it is a canonical basis in \eqref{eq:canonical-def}.
This construction is over a splitting field and does not identify an induction lattice.
\end{proof}

\subsection{Gelfand-pair commutants and spherical functions}

Let \(K\) be a splitting field for the generic Hecke algebra.
If a natural module attached to a term of the classification has
irreducible support of size \(r\), then Tits deformation and
Proposition~\ref{prop:Gelfandcriteria} give
\[
 \operatorname{End}_{\HH_K(W)}(M_K)\cong K^r.
\]
Thus the generic commutant and the finite corner are split commutative
algebras of the same dimension \(r\).  This is only a semisimple
comparison: it neither constructs a flat integral deformation of the
finite corner nor provides an integral basis or structure constants for
the generic commutant.

For an ordinary parabolic pair \((W,W_J)\), there is a more concrete
construction.  Let \(e_J\) be the normalized idempotent of
\(\HH_K(W_J)\) affording its trivial representation.  The algebra
\(e_J\HH_K(W)e_J\) is commutative exactly when the parabolically
induced module is MF; all such commutative parabolic Hecke algebras are
classified in \cite{APVM}.  Yanagida's Theorem~3.3 constructs normalized
Hecke-spherical functions that form a basis of this corner algebra;
their specialization at \(v=1\) gives the ordinary zonal spherical
functions \cite{Yanagida}.  Consequently, whenever a pair recorded in the
preceding type-by-type corollaries has a standard parabolic stabilizer,
it admits this additional spherical description.

For a general QP subgroup \(H\), there need not be an Iwahori--Hecke
subalgebra \(\HH(H)\subseteq\HH(W)\).  The finite corner
\(e_\lambda\C W e_\lambda\) therefore cannot simply be replaced by
the parabolic formula.  The QP module and bar operator provide the
appropriate deformation in this paper, while constructing a
distinguished integral basis of its commutant remains a separate
problem.  The type \(B_3\) obstruction below shows concretely that
commutativity of the complex commutant does not force an integral
\(W\)-graph on the natural induction lattice.

In the dihedral case the irreducible graphs can be written without an existence argument.
For \(\rho_h\), put \(c_h=2\cos(\pi h/m)\).
Take vertices \(x,y\), labels \(\tau(x)=\{s\},\tau(y)=\{t\}\), and both edge weights \(c_h\).
Then
\[
 T_s=\begin{pmatrix}-v^{-1}&0\\c_h&v\end{pmatrix},
 \qquad
 T_t=\begin{pmatrix}v&c_h\\0&-v^{-1}\end{pmatrix}.
\]
The product has determinant \(1\) and trace
\(c_h^2-2=2\cos(2\pi h/m)\); the matrices satisfy the length-\(m\) braid relation and specialize to \(\rho_h\).
Together with the scalar graphs for \(\lambda_{ab}\), their disjoint unions explicitly realize every family of Theorem~\ref{thm:dihedral}.
The bar fixes \(x,y\).

For \(H_3\), the first three families of Theorem~\ref{thm:H3} have perfect realizations.
Replacing the two linear summands by \(M(W/\Ker\om)\) gives the fourth family's natural bar and canonical structure.
The two-point graph joins the corresponding two-dimensional term.

\begin{remark}
Preserving a natural lattice is a stronger requirement than realizing its complex character.
The next subsection gives explicit modules in which the natural lattice admits a compatible bar and a unique triangular canonical basis, but cannot underlie any integral \(W\)-graph.
This distinction is part of the realization theorem.
\end{remark}

\subsection{Natural canonical structures for the additional \texorpdfstring{\(B_3\)}{B3} models}
\label{sec:B3canonical}

This subsection completes the natural realization problem for all the characters in Table~\ref{tab:B3terms}.
Recall that the natural lattice of a term is
\[
 L(H,\sigma)=\Z W\otimes_{\Z H}\Z_{\sigma|_H}.
\]
Its basis consists of signed coset vectors; equality of complex characters alone does not identify this lattice.

\begin{theorem}\label{thm:B3natural}
All twenty-four type \(B_3\) models have termwise Hecke deformations of their natural lattices, with compatible bar operators and explicit canonical bases.
Among the sixteen additional models, precisely
\[
 \mathscr N_3,\ \mathscr N_4,\ \mathscr N_5,\qquad
 \om\mathscr N_3,\ \om\mathscr N_4,\ \om\mathscr N_5
\]
admit termwise natural integral canonical \(W\)-graphs.
The other ten contain a natural summand lattice which cannot be the specialization of any integral \(W\)-graph module.
Including the eight perfect models, the total number admitting termwise natural canonical \(W\)-graphs is fourteen.
\end{theorem}

\subsubsection{The ordinary and parabolic summands}

The characters \(F_1,F_\om,L_+\) are respectively the ordinary four-point \(M\), four-point \(N\), and two-point QP modules.
The terms \(C,\dd C,D,\dd D\) have ordinary \(M/N\) or parabolic \(M/N\) constructions.
Each \(B_\tau\) or \(E_\tau\) is induced from a one-dimensional Hecke character of the displayed standard parabolic subgroup; apply parabolic bar and \(W\)-graph induction.
For \(A_\tau\), inflate the three-point symmetric-group parabolic module: \(T_1,T_2\) have the ordinary \(S_3\) action with its sign choice, while \(T_0\) is the scalar \(v\) or \(-v^{-1}\) determined by \(\tau(s_0)\).
The length-four braid is automatic for this scalar action.
These modules have their usual bars and canonical graphs.
Finally the four linear characters have one-vertex graphs.
Thus only
\[
 F_\ee=\ee+R,\qquad F_\dd=\dd+\om R,\qquad L_-=\dd+\ee
\]
require a new natural construction.

\begin{example}[The four-point canonical graph]
For \(F_1\), the four standard heights are \(0,1,2,3\).
Its canonical labels and nonzero weights are
\[
 \tau(0)=\{s_0\},\quad\tau(1)=\{s_1\},\quad
 \tau(2)=\{s_0,s_2\},\quad\tau(3)=\emptypar,
\]
with weight one on \(0\leftrightarrow1\), \(1\leftrightarrow2\), and \(2\to3\).
The graph is shown in Figure~\ref{fig:B3graph}.
\end{example}

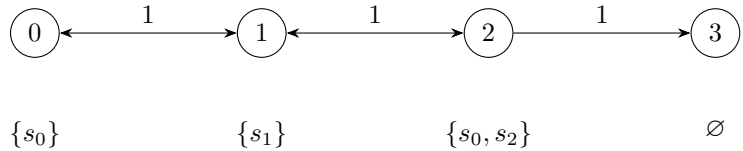
\begin{figure}[htbp]
\centering
\begin{tikzpicture}[>=Stealth,every node/.style={font=\small}]
\node[circle,draw] (a) at (0,0) {\(0\)};
\node[circle,draw] (b) at (3,0) {\(1\)};
\node[circle,draw] (c) at (6,0) {\(2\)};
\node[circle,draw] (d) at (9,0) {\(3\)};
\draw[<->] (a)--node[above]{\(1\)}(b);
\draw[<->] (b)--node[above]{\(1\)}(c);
\draw[->] (c)--node[above]{\(1\)}(d);
\node[below=7mm of a] {\(\{s_0\}\)};
\node[below=7mm of b] {\(\{s_1\}\)};
\node[below=7mm of c] {\(\{s_0,s_2\}\)};
\node[below=7mm of d] {\(\emptypar\)};
\end{tikzpicture}
\caption{The canonical \(W\)-graph for the natural four-point term \(F_1\), in convention \eqref{eq:Wgraph-def}.}
\label{fig:B3graph}
\end{figure}

\subsubsection{Two four-dimensional kernels}

Let \(u\) be either \(v\) or \(-v^{-1}\); thus \(u^2=1+au\) and \(\bar u=u^{-1}\).
Define
\[
 V_8(u)=\HH(W)\otimes_{\HH(P)}\A_u,\qquad
 V_4(u)=
 \begin{cases}M(W/K),&u=v,\\N(W/K),&u=-v^{-1},\end{cases}
\]
where \(T_1,T_2\) act on \(\A_u\) by \(u\).
The minimum vector of \(V_4(u)\) satisfies these two scalar relations.
Consequently there is a surjective Hecke homomorphism
\begin{equation}
 \pi_u:V_8(u)\longrightarrow V_4(u),\qquad m_P\longmapsto m_K .
 \label{eq:B3surjection}
\end{equation}
Both bars fix these minimum vectors, so the map commutes with bar.
Four standard vectors of \(V_8(u)\) map to the four standard vectors of \(V_4(u)\).
They give an \(\A\)-linear section, and
\[
 Q(u)=\Ker\pi_u
\]
is free of rank four.
The section is only asserted to be \(\A\)-linear.

At \(v=1\), the map for \(u=v\) identifies the eight sign vertices in antipodal pairs.
Its kernel, generated by the antipodal differences, is \(L(K,\ee)\).
The corresponding signed map for \(u=-v^{-1}\) has kernel \(L(K,\dd)\).
More explicitly, number the eight cosets in minimum-length order with the first four mapping to the four standard quotient vectors.
An ordered kernel basis \(\kappa_0,\ldots,\kappa_3\) specializes as
\[
 \kappa_i(1)=m_{4+i}(1)-m_{3-i}(1),\qquad 0\leq i\leq3 .
\]
These four vectors are a signed coset basis of the stated induced lattice.
Their absolute quotient action has QP heights \(0,1,2,3\).

The full Hecke action on this basis is
\begin{align}
 T_0\kappa_0&=\kappa_1,&
 T_0\kappa_1&=\kappa_0+a\kappa_1,\notag\\
 T_0\kappa_2&=\kappa_3,&
 T_0\kappa_3&=\kappa_2+a\kappa_3,\label{eq:B3T0}\\
 T_1\kappa_0&=u\kappa_0,&
 T_1\kappa_1&=\kappa_2,\notag\\
 T_1\kappa_2&=\kappa_1+a\kappa_2,&
 T_1\kappa_3&=u\kappa_3,\label{eq:B3T1}
\end{align}
and
\begin{align}
 T_2\kappa_0&=a\kappa_0-\kappa_1,&
 T_2\kappa_1&=-\kappa_0,\notag\\
 T_2\kappa_2&=u\kappa_2-a\kappa_0-au\kappa_1,&
 T_2\kappa_3&=u\kappa_3-au\kappa_0-au^2\kappa_1 .
 \label{eq:B3T2}
\end{align}
These formulas may alternatively be taken as explicit definitions of the two modules.
Reduction using \(u^2=1+au\) verifies each quadratic relation and the braids of lengths \(4,3,2\).
The kernel construction proves the natural-lattice identification, and accounts for the correction terms in \eqref{eq:B3T2}.

\begin{proposition}\label{prop:B3kernelcanonical}
The bars on \(Q(u)\) are
\begin{equation}
 \bar\kappa_0=\kappa_0,\qquad
 \bar\kappa_i=\kappa_i-a\sum_{j=0}^{i-1}u^{j+1-i}\kappa_j
 \quad(1\leq i\leq3).
 \label{eq:B3kernelbar}
\end{equation}
For \(u=v\), the canonical basis is
\begin{equation}
 C_i=\sum_{j=0}^i v^{j-i}\kappa_j .
 \label{eq:B3kernelCplus}
\end{equation}
For \(u=-v^{-1}\), it is
\begin{equation}
 C_0=\kappa_0,\qquad C_i=\kappa_i+v^{-1}\kappa_{i-1}\quad(1\leq i\leq3).
 \label{eq:B3kernelCminus}
\end{equation}
\end{proposition}
\begin{proof}
Restrict the bar in \eqref{eq:B3surjection}, or substitute \eqref{eq:B3kernelbar} into the twelve generator formulas.
One obtains \(\bar{\bar\kappa_i}=\kappa_i\) and
\(\overline{T_s\kappa_i}=(T_s-a)\bar\kappa_i\).
Substitution in \eqref{eq:B3kernelCplus} and \eqref{eq:B3kernelCminus} gives bar-fixed vectors.
Their leading terms and negative-power triangularity are immediate.
Lemma~\ref{lem:canonicalalgorithm} gives uniqueness.
\end{proof}

\begin{example}[The last canonical vector]
For \(F_\ee\), the last vector is
\[
 C_3=\kappa_3+v^{-1}\kappa_2+v^{-2}\kappa_1+v^{-3}\kappa_0.
\]
For \(F_\dd\), it is \(C_3=\kappa_3+v^{-1}\kappa_2\).
Both are canonical relative to their stated bars, despite the different lower supports.
\end{example}

\subsubsection{The two-dimensional signed term}

For \(L_-\), on \(m_0,m_1\), take
\begin{equation}
 A=\begin{pmatrix}0&1\\1&a\end{pmatrix},\qquad
 T_0=A,\qquad T_1=T_2=aI-A .
 \label{eq:B3two}
\end{equation}
The matrices satisfy their quadratic relations.
They commute and \(A(aI-A)=-I\), proving the length-four and length-two braids; the remaining braid is immediate since \(T_1=T_2\).
At \(v=1\), \(s_0\) exchanges the two basis vectors and \(s_1=s_2=-s_0\).
This is exactly \(L(W^+,\dd)\) in the coset basis \(1,s_0\).
The bar and canonical basis are
\begin{equation}
 \bar m_0=m_0,\qquad \bar m_1=m_1-am_0,\qquad
 C_0=m_0,\qquad C_1=m_1+v^{-1}m_0 .
 \label{eq:B3twobar}
\end{equation}
Direct substitution proves all the required identities.
For \(L_+\), all three generators equal \(A\), giving the ordinary two-point graph.

\subsubsection{An intrinsic integral \texorpdfstring{\(W\)}{W}-graph obstruction}

\begin{lemma}[Commuting-reflection rank bound]\label{lem:rankobstruction}
Let \(s,t\) be commuting simple reflections.
Suppose an integral \(W\)-graph module specializes to a lattice \(L\).
Let \(n_{++},n_{+-},n_{-+},n_{--}\) be the dimensions over \(\Q\) of the joint eigenspaces of \(s,t\) on \(\Q\otimes L\).
If \(n_{--}=0\), then
\[
 \rank_{\F_2}(s-I)\leq n_{++}.
\]
If \(n_{++}=0\), then
\[
 \rank_{\F_2}(s-I)\leq n_{--}.
\]
\end{lemma}
\begin{proof}
Use the transposed \(W\)-graph convention in which a vertex carrying \(s\) has scalar \(-v^{-1}\), while a vertex not carrying \(s\) has scalar \(v\) plus edges to vertices carrying \(s\).
This transposes the matrices from \eqref{eq:Wgraph-def} and leaves the ranks and joint traces unchanged.
Group the vertices by their \(s,t\) labels \(00,10,01,11\).

An edge from \(01\) to \(10\) has weight \(b\).
At this matrix entry the commutation relation \(T_sT_t=T_tT_s\) gives
\(-v^{-1}b=vb\), so \(b=0\).
There is no extra contribution through a third vertex at this entry: after an \(s\)-edge leaves \(01\), it reaches \(10\) or \(11\), and a subsequent \(t\)-edge cannot end at \(10\); the reversed product has the analogous restriction.
The reverse edge is excluded identically.
After these entries vanish, the two generator matrices are simultaneously triangular with respect to inclusion of the two-bit labels.
Their diagonal entries at \(v=1\) give the four joint eigenvalue multiplicities.
Thus the vertex counts are the stated \(n_{\pm\pm}\).

Modulo two, every diagonal entry of \(s-I\) vanishes.
If no \(11\) vertex exists, its only possible nonzero block is from \(00\) to \(10\), whose rank is at most \(n_{++}\).
If no \(00\) vertex exists, its only possible nonzero block is from \(01\) to \(11\), of rank at most \(n_{--}\).
This proves both assertions.
\end{proof}

Apply the lemma to \(s_0,s_2\) in \(B_3\).
The natural signed-coset matrices give
\[
\begin{array}{c|c|c|c}
\text{lattice}&(n_{++},n_{+-},n_{-+},n_{--})&
 \rank_{\F_2}(s_0-I)&\text{bound}\\ \hline
L(W^+,\dd)&(0,1,1,0)&1&0\\
L(K,\ee)&(1,1,2,0)&2&1\\
L(K,\dd)&(0,2,1,1)&2&1 .
\end{array}
\]
Each row violates the bound.
For example, on \(L(K,\ee)\), \(s_0\) permutes the four cosets in two pairs, so its rank modulo two is two.
The joint minus-minus space is zero and the plus-plus space is one-dimensional.
The joint dimensions can be recovered without choosing rational eigenvectors:
\[
 n_{\epsilon\eta}=\tfrac14\bigl(\rank_\Z L+
 \epsilon\,\tr(s)+\eta\,\tr(t)+\epsilon\eta\,\tr(st)\bigr),
 \qquad\epsilon,\eta\in\{1,-1\}.
\]
The signed-coset matrices in \cite{ZhangData} verify these four traces and each rank by exact arithmetic.
These traces and the ranks modulo two are invariant under integral changes of basis.
Thus no choice of Hecke deformation or integral \(W\)-graph basis can remove the obstruction.

\begin{proposition}\label{prop:intrinsicB3}
Any allowed realization of a character \(F_\ee,F_\dd,L_-\) has the respective natural lattice used above, up to integral \(W\)-isomorphism.
In particular the obstruction depends only on these inducing terms, not on the selected subgroup representatives.
\end{proposition}
\begin{proof}
Each degree is two or four, below the smallest index six of a proper parabolic subgroup of \(B_3\).
Thus the source must be the full group.
If an induced linear character contains a linear constituent \(\tau\), Frobenius reciprocity forces its inducing character on \(H\) to be \(\tau|_H\).
For \(L_-=\dd+\ee\), the index-two subgroup is therefore \(\Ker(\dd\ee)=W^+\).
For either \(F\) character, divide by its unique linear constituent.
The resulting permutation character is \(1+\ee R\), whose kernel is \(\langle z\rangle\).
The kernel of a transitive permutation character is the intersection of its point stabilizers, so \(\langle z\rangle\leq H\).
Since \(W/\langle z\rangle\cong S_4\), the image of \(H\) has order six.
Every order-six subgroup of \(S_4\) has a normal subgroup of order three and is contained in its order-six normalizer, a point stabilizer \(S_3\).
Thus the image of \(H\) is conjugate to such a stabilizer.
Hence \(H\) is conjugate to \(K\).
Conjugation identifies the corresponding induced lattices over \(\Z W\).
\end{proof}

\begin{proof}[Proof of Theorem~\ref{thm:B3natural}]
The ordinary and parabolic constructions, the two kernels, and \eqref{eq:B3two} supply natural Hecke, bar, and canonical structures for all twenty-six characters used in the twenty-four families.

The three families \(\mathscr N_3,\mathscr N_4,\mathscr N_5\) use only the twenty-three terms with ordinary or parabolic canonical graphs.
Their graphs are the disjoint unions of the summand graphs.
The same is true of their Coxeter-sign duals, using the appropriate dual canonical constructions.
Each of the other five representatives in Table~\ref{tab:B3extra}, and each of its duals, contains \(F_\ee,F_\dd\), or \(L_-\).
Lemma~\ref{lem:rankobstruction} and Proposition~\ref{prop:intrinsicB3} rule out an integral graph on that summand lattice.
This gives exactly six positive and ten negative additional families.
The eight perfect families have their natural canonical graphs, so the total is fourteen.
\end{proof}

\begin{example}[Canonical vectors without an integral graph]
In either four-dimensional kernel,
\[
 T_2C_0=vC_0-C_1.
\]
Thus the canonical vectors do not satisfy the scalar-positive case of
\eqref{eq:Wgraph-def}.
The rank argument above is stronger: it excludes every integral \(W\)-graph basis on this natural lattice, independently of this particular canonical choice.
Nonetheless \eqref{eq:B3kernelbar}--\eqref{eq:B3kernelCminus} give its complete bar and canonical basis.
For the two-dimensional term \(L_-\), Corollary~\ref{cor:B3Gelfand}
gives the commutative finite corner algebra \(\C^2\), while the first
row of the rank-obstruction table excludes an integral \(W\)-graph on
its natural signed-coset lattice.  Thus even the smallest genuinely
twisted term separates the finite Gelfand property from the integral
canonical-graph property.
\end{example}

The negative assertion is termwise: it preserves the direct sum into the specified natural inducing lattices.
The theorem does not identify this requirement with a change of lattice after decomposing the complex character.
Over a splitting field, Corollary~\ref{prop:allHecke} still gives a \(W\)-graph for every one of the twenty-four characters and their prescribed character summands.
\medskip

\noindent\textbf{Declaration of generative AI use.}
During the preparation of this work, OpenAI ChatGPT and Codex were used to assist with literature searches, proof drafting and revision, the development of computational checks, and manuscript preparation.
The computational methods and their verification records are identified in the text and the accompanying archive.
The author is responsible for the content of the submitted manuscript.

\end{document}